\documentclass{amsart}

\usepackage[T1]{fontenc}
\usepackage[latin1]{inputenc}
\usepackage{amsmath}
\usepackage{amssymb}
\usepackage{amsthm}
\usepackage{dsfont}
\usepackage{color}
\theoremstyle{plain}\newtheorem{theo}{Theorem}
\theoremstyle{plain}\newtheorem{cor}[theo]{Corollary}
\theoremstyle{definition}\newtheorem{rem}[theo]{Remark}
\theoremstyle{plain}
\theoremstyle{plain}\newtheorem{lem}[theo]{Lemma}
\theoremstyle{plain}\newtheorem{prop}[theo]{Proposition}
\theoremstyle{definition}
\theoremstyle{definition}\newtheorem{hypo}[theo]{Assumption}

\newcommand{\N}{{\mathds N}}

\newcommand{\Z}{{\mathds Z}}
\DeclareMathOperator{\var}{Var}
\DeclareMathOperator{\cov}{Cov}
\DeclareMathOperator{\sgn}{sgn}

\begin{document}

\title[Convergence of U-statistics to integrals of L\'evy sheet]{Convergence of U-statistics indexed by a random walk to
stochastic integrals of a L\'evy sheet}

\author{Brice Franke}
\address{Universit\'e de Brest,
UMR CNRS 6205, Laboratoire de Math\'ematique de Bretagne Atlantique,
6 avenue Le Gorgeu, 29238 Brest cedex, France}
\email{Brice.Franke@univ-brest.fr}

\author{Fran\c{c}oise P\`ene}
\address{Universit\'e de Brest,
UMR CNRS 6205, Laboratoire de Math\'ematique de Bretagne Atlantique,
6 avenue Le Gorgeu, 29238 Brest cedex, France}
\email{francoise.pene@univ-brest.fr}

\author{Martin Wendler}
\address{Ruhr-Universit\"at Bochum, Fakult\"at f\"ur Mathematik, Universit\"atsstra\ss e 150, 44801 Bochum, Germany}
\email{Martin.Wendler@rub.de}

\keywords{random walk; random scenery; $U$-statistics; stable limits, L\'evy sheet\\
Fran\c{c}oise P\`ene is supported by the french ANR project MEMEMO2 (ANR-10-BLAN-0125).}

\begin{abstract}
A $U$-statistic indexed by a $\Z^{d_0}$-random walk $(S_n)_n$ is a process $U_n:=\sum_{i,j=1}^n
h(\xi_{S_i},\xi_{S_j})$ where $h$ is some real-valued function and $(\xi_k)_k$ is a sequence of iid random variables, which are  
independent of the walk. 
Concerning the walk, we assume either that it is transient or that its increments are in the normal domain of attraction of a strictly stable
distribution of exponent $\alpha\in[d_0,2)$. 
We further assume that the distribution of $h(\xi_1,\xi_2)$ belongs to the normal domain of attraction of a strictly stable distribution of exponent
$\beta\in(0,2)$. For a suitable renormalization $(a_n)_n$ we establish the convergence in distribution of the sequence of processes $(U_{\lfloor nt\rfloor}/a_n)_t;n\in\N$ 
to some suitable observable of a L\'evy sheet $(Z_{s,t})_{s,t}$. 
The limit process is the diagonal process $(Z_{t,t})_t$ when the walk is transient or when $ \alpha= d_0 $. When $ \alpha>d_0=1 $ the limit process is some stochastic integral 
with respect to $Z$.
\end{abstract}

\maketitle

\section{Introduction}

Given a random walk $(S_n)_{n\ge 0}$ on $\mathbb Z^{d_0}$ and a sequence of independent identically
distributed (iid) real random variables $(\xi_k)_{k\in\mathbb Z^{d_0}}$, independent
one from each other, one can consider the random walk in random scenery
$\mathcal S_n:=\sum_{k=1}^n\xi_{S_k}$. In particular one is interested in the limit behavior of the sequence of renormalized processes
$ (\nu_n^{-1}\mathcal S_{[nt]})_{t\geq0};n\in\N $.
In this context the following assumptions are usually made:

\begin{itemize} 
\item[(A)] either $ S_n $ is transient or there exists some  $\alpha\in[d_0,2]$ such that  $n^{-\frac 1\alpha}S_n; n\in\N$ converges in distribution to 
a random variable;
\item[(B)] $n^{-\frac 1\beta}\sum_{k=1}^n\xi_{k};n\in\N $
converges in distribution to 
a random variable for some $\beta\in(0,2]$.
\end{itemize}
 
Note that in the case  $\alpha>{d_0}=1$ the assumption (A) implies that the sequence of stochastic processes
$(n^{-\frac 1\alpha}S_{\lfloor nt\rfloor})_{t>0};n\in\N$ converges in distribution to
some $\alpha$-stable L\'evy process $(Y_t)_{t>0}$ which admits a local time $(\mathcal L_t(x),\ t\ge 0,\ x\in\mathbb R)$.
Similarly, assumption (B) implies that  $(n^{-\frac 1\beta}\sum_{k=1}^{\lfloor nt\rfloor}\xi_{k})_{t>0};n\in\N $ 
converges in distribution to some $\beta$-stable process $(\mathcal Z_t)_{t>0}$.\footnote{to simplify notations, for every $k\in\mathbb Z$, we write
$\xi_k$ for $\xi_{(k,...,k)}$}
Subsequently we will use  $(\mathcal Z_{-t})_{t>0}$ to denote an independent copy of $(\mathcal Z_t)_{t>0}$.

Random walks in random scenery have been studied by many authors since the early works of Borodin \cite{Borodin1,Borodin2} and Kesten and Spitzer \cite{kest}.
In particular, \cite{Bolthausen,DU,cast} complete the study of the limit
in distribution of random walks in random scenery. 
The asymptotic behavior of the sequence $ (\nu_n^{-1}\mathcal S_{\lfloor nt\rfloor})_{t>0};n\in\N$ is summarized in the following table
(where $d_1$ and $d_2$ are explicit constants depending on $(S_n)$ and on $\beta$):

$$\begin{array}{|c|c|c|c|}
\hline
\mbox{Cases}&\mbox{normalization}&\mbox{Limit process}&\mbox{Space of convergence in distribution}\\
\hline
\mbox{transient}&\nu_n:=n^{\frac 1\beta}&(d_1\mathcal Z_t)_t&
\begin{array}{c}\mbox{finite distributions}\\
\mbox{if }\beta\ne 1\mbox{:Skorokhod space with }M_1\mbox{-metric}
\end{array}\\
\hline
\alpha={d_0}&\nu_n:=n^{\frac 1\beta}(\log n)^{1-\frac 1\beta}&(d_2\mathcal Z_t)_t&\begin{array}{c}\mbox{finite distributions}\\
\mbox{if }\beta\ne 1\mbox{:Skorokhod space with }M_1\mbox{-metric}
\end{array}\\
\hline
\alpha>{d_0}&\nu_n:=n^{1-\frac 1\alpha+\frac 1{\alpha\beta}}&(\Delta_t:=\int_{\mathbb R^*}\mathcal L_t(x)\,
d{\mathcal Z}_x)_t&\mbox{Skorokhod space with }J_1\mbox{-metric}\\
\hline
\end{array}$$

In this paper we want to do a similar investigation for U-statistics indexed by a random walk. 
To introduce the objects let $ E $ be some measurable set and  $(\xi_k)_{k\in\mathbb Z^{d_0}}$ an  iid  sequence of $E$-valued random
variables. Often we might abbreviate this family of random variables by $ \xi $ and call it the scenery.
Moreover, let $(S_n)_{n\ge 1}$ be as above a random walk on $ \mathbb{Z}^{d_0} $, which is independent of the scenery $\xi$. 
We will also use the short notation $ S $ for the random walk. 
For some measurable function $ h:E^2\rightarrow\mathbb R $, we consider the $U$-statistic indexed by $S$ defined through
$$U_n:=\sum_{i,j=1}^nh(\xi_{S_i},\xi_{S_j}).$$

We are interested in results of distributional convergence for $(U_n)_n$ (after some suitable normalization)
under the assumption that the distribution of $h(\xi_1,\xi_2)$ is in the normal domain of attraction of a $\beta$-stable distribution.
Let us assume without loss of generality that $h$ is symmetric.

If $ \beta>1 $ we can introduce $\vartheta_k:=\mathbb E[h(\xi_0,\xi_k)|\xi_0]$. 
Two different situations can occur. We will say that the kernel is degenerate if $\vartheta_1= 0$ almost surely. Otherwise, we will say that the kernel is 
non-degenerate.

The case when when $h(\xi_1,\xi_2)$ is square integrable and centered (which implies  $\beta=2$) has been fully studied by Guillotin-Plantard and her co-authors. 
In this case only two kind of behaviors can occur:
\begin{itemize}
\item[(a)] the kernel is non-degenerate, then one can use Hoeffding decomposition to show that $U_n$ behaves essentially as
$\sum_{i,j=1}^n(\vartheta_{S_i}+\vartheta_{S_j})
=2n\sum_{i=1}^n\vartheta_{S_i}$.
\item[(b)] the kernel is degenerate, then Hilbert-Schmidt theory can be used to represent the kernel as $h(x,y)=\sum_p\lambda_p\phi_p(x)\phi_p(y)$ 
and to show that $U_n$ behaves as $\sum_p\lambda_p(\sum_{i=1}^n\phi_p(S_i))^2$.
\end{itemize}
This has been proved by Cabus and Guillotin-Plantard in
\cite{cab} for random walks in $ \Z^{d_0} $ with $ d_0\geq 2 $  and by 
Guillotin-Plantard and Ladret in \cite{guil} for random walks in $ \Z $. 

Note that the situation treated in \cite{cab} splits into the case $ d_0>2 $, where the walk is transient, and the singular case $ d_0=2 $, where the random walk is null recurrent. 
However, in this last case the limit process $ (Y_t)_{t\geq0} $ does not have local time.  
In contrast to this the assumptions made in \cite{guil} correspond to some null recurrent random walk with existing local time for $ (Y_t)_{t\geq0} $; i.e.:  $\alpha>d_0=1$.

The special form of the representations given in (a) and (b) implies that for $\beta=2$, the study of $(U_n)_n$ 
can be reduced to the study of some suitable random walk in random scenery (either $\sum_{i=1}^n\vartheta_{S_i}$ or $\sum_{i=1}^n\phi_p(S_i)$). 
Thus the limits can be expressed in terms of processes which already occurred in the random scenery situation. 

In the transient case or if $ d_0= 2 $  the limit process turns out to be Brownian motion $ (B_t)_{t\geq0} $ when the kernel is non-degenerate. 
In the degenerate situation the limit has the representation $ \sum_p\lambda_p (B_t^{(p)})^2 $, where $ (B_t^{(p)})_{t\geq0};p\in\N $ is a sequence of independent
Brownian motions (see \cite{cab}).  

If on the other hand $\alpha>d_0=1$, then  in the non-degenerate situation the limit is the usual process $\Delta_t:=\int_{\mathbb R^*}\mathcal L_t(x)\,dB_x$, where $(B_x)_{x>0}$ and  $(B_{-x})_{x>0}$ are independent one-dimensional Brownian motions. 
In the degenerate case the limit takes the form $ \sum_p\lambda_p\big(\int_{\mathbb R^*}\mathcal L_t(x)\,dB^{(p)}_x\big)^2 $, where the 
pairs $ (B^{(p)}_x)_{x>0} $, $ (B^{(p)}_{-x})_{x>0} $ form a sequence of independent copies of the pair  $(B_x)_{x>0}$, $(B_{-x})_{x>0}$ (see \cite{guil}). \\[-2mm]

Let us further mention that (a) includes the case where $h(x,y)=g(x)+g(y)$ and that (b) includes the case when $h(x,y)=g(x)g(y)$. 
Here $g:E\rightarrow\mathbb{R}$ is a measurable function such that $g(\xi_1)$ is square integrable and centered.\\[-2mm]

When $ 1<\beta<2$, a similar behavior can occur in the non-degenerate case. 
For instance, in \cite{BFMa}, we use Hoeffding 
decomposition to prove the following:
\begin{itemize}
\item[(a')] If $1<\beta\le 2$ and if the distribution of $\vartheta_1$ is
in the normal domain of attraction of a $\beta$-stable distribution
then $U_n$ behaves as $2n\sum_{i=1}^n\vartheta_{S_i}$.
\end{itemize}
This holds for example if $h(x,y)=g(x)+g(y)$.
The limit then turns out to be $ \beta $-stable L\'evy process $ (Z_t)_{t\geq0} $ when the walk is transient or when $ \alpha= d_0 $. 
However, when $ \alpha>d_0$  the limit has the representation $\Delta_t:=\int_{\mathbb R^*}\mathcal L_t(x)\,dZ_x$, where $(Z_x)_{x>0}$ and  $(Z_{-x})_{x>0}$ are independent one-dimensional $ \beta $-stable L\'evy-motions (see \cite{BFMa}). 

On the other hand in the degenerate case, when $ \vartheta_1=0 $, different limits than those described in (b) can arise when $ 0<\beta<2 $. This is the purpose of the present paper.
The limit we obtain is the diagonal process $ (Z_{(t,t)})_{t\geq0} $ of a L\'evy sheet $ (Z_{t,s})_{t,s\geq0} $, when the walk is transient or when $\alpha=d_0$, and a stochastic integral 
$ \int_{\mathbb{R}^2}L_t(x)L_t(y)dZ_{x,y} $ with respect to four independent copies of the L\'evy sheet introduced above, when $\alpha>d_0$. 
These limits can be understood as two-dimensional analogues of the known limits for random walk in random scenery found by Kesten and Spitzer (see \cite{kest}).

To be more precise, let us keep assumption {\bf (A)} but replace {\bf (B)} on $(\xi_k)_k$ by the following assumption 
on $(h(\xi_k,\xi_\ell))_{k,\ell}$:

\begin{itemize}
\item[(B')] $(n^{-\frac 1\beta}\sum_{k=1}^nh(\xi_{2k},\xi_{2k+1}))_n $
converges in distribution to 
a random variable  with $\beta\in(0,2)$.
\end{itemize}

This implies that if $(h_{i,j})_{i,j}$ is a sequence of iid random variables
with the same distribution as $h(\xi_1,\xi_2)$, then the sequence of stochastic processes 
$(n^{-\frac 2\beta}\sum_{k=1}^{\lfloor nt\rfloor}\sum_{\ell=1}^{\lfloor ns\rfloor}
h_{i,j})_{t>0};n\in\N $ converges in law to  some $\beta$-stable L\'evy sheet  $(Z_{s,t})_{s,t>0}$ (which we extend on $\mathbb R^2$). 

In the present paper, under assumption (B') and some additional assumptions, we prove limit theorems for the $U$-statistic
which are summarized in the following table:

$$\begin{array}{|c|c|c|c|}
\hline
\mbox{Cases}&\mbox{normalization}&\mbox{Limit process}&\mbox{Space of convergence in distribution}\\
\hline
\mbox{transient}&\nu_n^2=n^{\frac 2\beta}&(d_1^2Z_{t,t})_t&
\mbox{finite distribution}\\
\hline
\alpha={d_0}&\nu_n^2=n^{\frac 2\beta}(\log n)^{2-\frac 2\beta}&(d_2^2Z_{t,t})_t&\mbox{finite distribution}\\
\hline
\alpha>{d_0}&\nu_n^2=n^{2-\frac 2\alpha+\frac 2{\alpha\beta}}&(\int_{\mathbb R^2}\mathcal L_t(x)\mathcal L_t(y)\, dZ_{x,y})_t&\mbox{Skorokhod space with }J_1\mbox{-metric}\\
\hline
\end{array}$$
\vspace{1mm}

The present paper is organized as follows. The assumptions and main results are stated in Section \ref{results}.
We give some examples which satisfy our assumptions  in Section \ref{examples}. We prove our results concerning convergence 
of finite distribution in Section \ref{finitedistrib}. In the spirit of \cite{DDMS},
our proof relies on the convergence of a suitably defined point process to a Poisson point process which is 
established by the use of Kallenberg theorem.
In Section \ref{tightness}, we prove the tightness for the $J_1$-metric when $\alpha>d_0$.
We complete our article with some facts on the $\beta$-stable L\'evy sheet $Z$ in Appendix \ref{calculsto}. In particular 
a construction of stochastic integrals with respect to $Z$ is given.

\section{Main results}\label{results}

Let $ (\Omega,\mathcal F,\mathbb{P}) $ be a suitable probability space and 
let $S=(S_n)_{n\ge 0}$ be a $\mathbb Z^{d_0}$-valued random walk on $  (\Omega,\mathcal F,\mathbb{P}) $ with
$S_0=0$ such that one of the following conditions holds:

\begin{itemize}
\item the random walk $(S_n)_{n\ge 0}$ is transient,
\item the random walk $(S_n)_{n\ge 0}$ is recurrent and there exists
$\alpha\in[d_0,2]$ such that $(n^{-{\frac 1\alpha}}S_n)_{n\ge 1}$ converges in distribution to
a random variable $Y$. In this case we further assume that 
$ \forall x\in\mathbb{Z}^{d_0}, \exists n\in\mathbb{N} : \mathbb{P}(S_n=x)>0 $.
\end{itemize}

Recall that, in the second case, $(n^{-{\frac 1\alpha}}S_{\lfloor nt\rfloor})_{t>0};n\in\N$ converges in distribution to 
an $\alpha$-stable process $(Y_t)_{t>0}$  such that $Y_1$ has the same law as $Y$.

In order to get a uniform notation for the different situations, we define $ \alpha_0 $ to be a number, which is one when the random walk is transient, and which takes the value 
$ \frac{\alpha}{d_0} $ in the recurrent case. 

Let $\xi=(\xi_\ell)_{\ell\in\mathbb Z^{d_0}}$ be a family of iid random variables on $  (\Omega,\mathcal F,\mathbb{P}) $  with values in some measurable
space $E$.
We assume that the two families $S$ and $\xi$ are independent.
Let $h:E\times E\rightarrow \mathbb R$ be a measurable function.
We are interested in the properties of the U-statistics process
$U_n:=\sum_{i,j=1}^nh(\xi_{S_i},\xi_{S_j})$.
In this work, we assume moreover that the following properties are satisfied.

\begin{hypo}\label{HYP}Let $\beta\in(0,2)$.
\begin{itemize}
\item[(i)] for every $x\in E$, $h(x,x)=0$;
\item[(ii)] $h$ symmetric (i.e. $h(x,y)=h(y,x)$ for every $x,y\in E$);
\item[(iii)] there exist $c_0,c_1\in[0,+\infty)$ with $c_0+c_1>0$ 
such that
\begin{equation}\label{hyph1}
\forall z>0,\ \ 
\mathbb P(h(\xi_1,\xi_2)\ge z)=z^{-\beta} L_0(z),\ \ \mbox{with}\ \ 
\lim_{z\rightarrow+\infty} L_0(z)=c_0;
\end{equation}
and
\begin{equation}\label{hyph1b}
\forall z>0,\ \ 
\mathbb P(h(\xi_1,\xi_2)\le -z)=z^{-\beta}L_1(z),\ \ \mbox{with}\ \ 
\lim_{z\rightarrow+\infty}L_1(z)=c_1;
\end{equation}
\item[(iv)] there exist $C_0>0$ and $\gamma>\frac{3\beta}{4}$ such that
\begin{equation}\label{hyph2}
\forall z,z'\in(0,+\infty),\ \ \ 
\mathbb P\Big(
    |h(\xi_1,\xi_2)|\ge z\ \mbox{and}\  |h(\xi_1,\xi_3)|\ge z'\Big)\le C_0\big(\max(1,z)\max(1,z')\big)^{-\gamma};
\end{equation}
\item[(v)] If $\beta>1$, then $\mathbb E[h(\xi_1,\xi_2)]=0$;
\item[(vi)] If $\beta\ge 4/3$, 
there exists $C'_0>0$ and $\theta'>\frac{3\beta}4-1$ such that
$$\forall M,M'\in(0,+\infty),\ \
\left|{\mathbb E}\left[{\mathbf h}_M(\xi_1,\xi_2){\mathbf h}_{M'}(\xi_1,\xi_3)\right]\right|\le
C'_0(MM')^{-\theta'},$$
where
${\mathbf h}_M(x,y):=h(x,y)\mathbf 1_{\{|h(x,y)|\le M\}}+\frac{\beta}{\beta-1}(c_0-c_1)M^{1-\beta}$.
\item[(vii)] If $\beta=1$, then $c_0=c_1$ and $\lim_{M\rightarrow +\infty}\mathbb E[h(\xi_1,\xi_2)
\mathbf 1_{\{|h(\xi_1,\xi_2)|\le M\}}]=0$.
\end{itemize}
\end{hypo}

Some examples satisfying the above assumptions are presented in the next section.

\begin{rem} \label{RQE0}
The following comments on the different points in  Assumptions  \ref{HYP} might be of some help:
\begin{itemize}
\item  Item (i) can be relaxed as will be proved in  Proposition \ref{relax} below.
\item  Item (ii) is not restrictive since one can always 
replace $h(z,z')$ by $(h(z,z')+h(z',z))/2$ without changing the sequence $(U_n)_n$.
\item Note that Item (iv) is a condition which ensures that the tail behavior resulting from coupling of the pairs $ (\xi_1,\xi_2) $  and 
        $ (\xi_1,\xi_3) $ does not interfere with the tail behavior of the single terms $ h(\xi_1,\xi_2) $.  A condition with the same spirit is condition (2.1) in \cite{DDMS}.
\item If Item (iii) holds and if for every $x\in E$ the distribution of $h(x,\xi_1)$ is symmetric, then Item (vi) and Item (vii) are also satisfied. 
Indeed, in this case, $c_0=c_1$ and
$$ \ \ \ \ \ \ {\mathbb E}\left[{\mathbf h}_M(\xi_1,\xi_2){\mathbf h}_{M'}(\xi_1,\xi_3)\right]
=\int_E\, {\mathbb E}[h(x,\xi_2)\mathbf 1_{\{|h(x,\xi_2)|\le M\}}]
{\mathbb E}[h(x,\xi_2)\mathbf 1_{\{|h(x,\xi_2)|\le M'\}}]\, d\mathbb P_{\xi_1}(x) =0.$$
\item Note that Item (iii) and Item (v) imply that the law of $ h(\xi_1,\xi_2) $ is in the domain of attraction of a $ \beta $-stable law for some $ \beta\in(0,2) $.
\end{itemize}
\end{rem}

Let $(h_{i,j})_{i,j}$ be a sequence of iid random variables with same
distribution as $h(\xi_1,\xi_2)$.
Observe that the Items (i), (iii), (v) and (vii) in Assumption \ref{HYP} describe the classical situation, where the sequence of random fields
$(n^{-\frac 2\beta}\sum_{i=1}^{\lfloor nx\rfloor}\sum_{j=1}^{\lfloor ny\rfloor}h_{i,j})_{x,y>0};n\in\N$ converges in law 
to a $\beta$-stable L\'evy sheet $ (\tilde Z_ {x,y})_{x,y\geq0}$ such that the characteristic function of $\tilde Z_{x,y}$ is given by
${\mathbb E}[e^{iz\tilde Z_{x,y}}]=\Phi_{xy(c_0+c_1),xy(c_0-c_1),\beta}(z)$, with

\begin{equation}\label{FonCar}
\Phi_{A,B,\beta}(z):=\exp\left(-|z|^\beta\int_0^{+\infty}\frac{\sin t}{t^\beta}\, dt\left(A-iB\sgn(z)
  \tan\frac{\pi\beta} 2\right)\right)\ \ \mbox{if}\ \beta\ne 1
\end{equation}
and
\begin{equation}\label{FonCar1}
\Phi_{A,B,1}(z):=\exp\Big(-|z|\left(\frac{\pi}{2} A+iB\sgn(z)\log|z|\right)\Big)
\end{equation}
(see \cite[p. 568-569]{Feller}).
In order to construct a continuation of the L\'evy sheet $ \tilde{Z} $ to all of $ \mathbb{R}^2 $ 
we use four independent copies $Z^{(\varepsilon,\varepsilon')}$ (with $\varepsilon,\varepsilon'\in\{1,-1\}$) of $\tilde Z$
to introduce $Z_{x,y}:=Z_{|x|,|y|}^{(\sgn(x),\sgn(y))}$ for all $ (x,y)\in\mathbb{R}^2 $.
In the following we will need to integrate some continuous compactly supported function $\psi$ with respect to $Z$, i.e.:
$$ \int_{\mathbb R^2}\psi(x,y)\, dZ_{x,y}.$$
More information on L\'evy sheets and on the construction of the integral can be found in Appendix \ref{calculsto}.

When $ \alpha>d_0=1 $, we assume moreover that $(Z_{x,y})_{x,y}$ is independent of the $\alpha$-stable process $(Y_t)_t$.

If the random walk is transient, we write $N_\infty$
for the total number of visits of the two sided random walk $(S_n)_{n\in\mathbb Z}$ to zero; i.e.: $N_\infty:=\sum_{n\in\mathbb Z}\mathbf{1}_{\{S_n=0\}}$.

\begin{theo}[Transient case]\label{THM0}
Suppose $(S_n)_{n\ge 0}$ is transient and Assumption \ref{HYP}. We set $a_n:=n^{\frac 2\beta}$. Then the finite distributions of $((U_{\lfloor nt\rfloor}/a_n)_{t>0})_n$
converge to the finite distributions of $(K_\beta^{\frac 2\beta} Z_{t,t})_{t>0}$,
with $K_\beta:={\mathbb E}[N_\infty^{\beta-1}]$.
\end{theo}

In particular the previous theorem holds for the deterministic $ \mathbb{Z} $-valued walk $ S_n=n $ (for which $K_\beta=1$). In that case our result
boils down to a result on classical U-statistics which was established by Dabrowski, Dehling, Mikosch and Sharipov in  \cite{DDMS}.
We emphasize this point in the following corollary, since the link to the L\'evy sheet was not mentioned  in \cite{DDMS}.

\begin{cor}[Deterministic case]\label{classicalUstat}
 Suppose Assumption \ref{HYP} and set $a_n:=n^{\frac 2\beta}$. The finite distributions of
 $((\sum_{i,j=1}^{\lfloor nt\rfloor}h(\xi_i,\xi_j)/a_n)_{t>0})_n$ converge to the
 finite distributions of $(Z_{t,t})_{t>0}$.
\end{cor}

As usual $\Gamma$ will stand for the Gamma function.
We also write $N_n(x)$ for the occupation time of $S$ at $x$ up to time $n$, i.e.:
$$N_n(x):=\sum_{i=1}^n\mathbf{1}_{\{S_i=x\}}.$$
We define the maximal occupation time
of $S$ up to time $n$ through $N_n^*:=\max_xN_n(x)$  and  the range of $S$ up to time $n$ by
$$R_n:=\#\{y\in\mathbb Z^{d_0}\, :\, N_n(y)>0\}.$$
We recall that, when $\alpha=d_0$, there exists $c_3>0$ such that 
\begin{equation}\label{range}
R_n\sim c_3n/\log n\ a.s.\ \ \mbox{as}\ \ n\rightarrow\infty.
\end{equation}

\begin{theo}[Recurrent case without local time]\label{THM00}
Suppose $\alpha=d_0\in\{1,2\}$ and Assumption \ref{HYP}. We set $a_n:=n^{\frac 2\beta}(\log n)^{2-\frac 2\beta}$. Then the finite distributions of 
$((U_{\lfloor nt\rfloor}/a_n)_{t>0})_n$
converge to the finite distributions of $(K_\beta^{\frac 2\beta} Z_{t,t})_{t>0}$,
with $K_\beta:=\Gamma(\beta+1)/c_3^{\beta-1}$ and with $c_3$ given by (\ref{range}).
\end{theo}

When $\alpha>d_0$ (which implies $d_0=1$), we prove a result of convergence in distribution
in the Skorokhod space for the $J_1$-metric. Recall that $\mathcal {\mathbf h}_M(x,y)=h(x,y)\mathbf 1_{\{|h(x,y)|\le M\}}+\frac{\beta}{\beta-1}(c_0-c_1)M^{1-\beta}$.

\begin{theo}[Recurrent case with local time]\label{THM1}
Assume $\alpha\in(1,2]$, $d_0=1$ and Assumption \ref{HYP}. We set $a_n:=n^{2\delta}$ with $\delta=1-\frac 1\alpha+\frac 1{\alpha\beta}$.
Then, for every $T>0$, $((U_{\lfloor nt\rfloor}/a_n)_{t\in[0,T]})_n$ converges in distribution (in the Skorokhod space $D([0,T])$ endowed with the $J_1$ metric) to
$(\int_{\mathbb R^2}\mathcal L_t(x)\mathcal L_t(y)\, dZ_{x,y})_{t\in[0,T]}$,
where $(\mathcal L_t(x),\ t\ge 0,\ x\in\mathbb R)$ is a jointly continuous version of 
the local time at point $x$ at time $t$ of $(Y_s)_{s\ge 0}$
(such that, for every $t$, $\mathcal L_t$ is compactly supported).
\end{theo}

Observe that, in every case, there exists $c>0$ such that
\begin{equation}\label{anequiv}
a_n\sim cn^2(\mathbb E[R_n])^{\frac 2\beta-2}
\end{equation}
(see for example \cite[p. 36]{Spitzer} and \cite[pp. 698-703]{LeGallRosen}).
It is worth noting that $U_n$ can be rewritten as follows
$$U_n=\sum_{x,y\in\mathbb Z^{d_0}}h(\xi_x,\xi_y)N_n(x)N_n(y) .$$

\begin{prop}\label{relax}
The results of convergence of finite dimensional distributions of
Theorems \ref{THM0}, \ref{THM00} and \ref{THM1} hold also if we replace Item (i) of Assumption \ref{HYP}
by the following assumption:\\[0.5mm]
(i') $\mathbb E[\exp(iu h(\xi_1,\xi_1))]-1=O(|u|^{\beta'})$ for some $ \beta'> \beta/2 $.
\end{prop}
Observe that (i') includes (i) and the case when 
$h(\xi_1,\xi_1)$ is in the normal domain of attraction of a $\beta'$-stable distribution for some $\beta'>\beta/2$, in particular this applies if $h(\xi_1,\xi_1)$ has the same distribution as $h(\xi_1,\xi_2)$.
\begin{proof}
Due to Theorems \ref{THM0}, \ref{THM00} and \ref{THM1}, we know that the finite dimensional distributions of
$$\left(\left(\sum_{x\ne y} h(\xi_x,\xi_y)N_{\lfloor nt\rfloor}(x)N_{\lfloor nt\rfloor}(y)/a_n\right)_{t>0}\right)_n$$ converge.
It remains to prove that $(\sum_x h(\xi_x,\xi_x)N_{\lfloor nt\rfloor}^2(x)/a_n)_n$ converges in probability to 0 (for every $t>0$). We write $\varphi_{h(\xi_1,\xi_1)}$ for
the characteristic function of $h(\xi_1,\xi_1)$.
Let $t>0$ and $u$ be two real numbers. We have
\begin{eqnarray*}
\mathbb E\left[\exp\left(iu\sum_{x\in\mathbb Z^{d_0}} \frac{h(\xi_x,\xi_x)N_{\lfloor nt\rfloor}^2(x)}{a_n}
\right)\right]&=&
\mathbb E\left[\prod_{x\in\mathbb Z^{d_0}}\varphi_{h(\xi_1,\xi_1)}\left(\frac{uN_{\lfloor nt\rfloor}^2(x)}{a_n}
\right)\right].
\end{eqnarray*}
To conclude we just have to prove that 
$\left(\prod_{x\in\mathbb Z^{d_0}}\varphi_{h(\xi_1,\xi_1)}\left(\frac{uN_{\lfloor nt\rfloor}^2(x)}{a_n}\right)\right)_n$ converges almost
surely to 1. Due to (i'), there exists $C_2>0$ such that we have
\begin{eqnarray*}
\left|\prod_{x\in\mathbb Z^{d_0}}\varphi_{h(\xi_1,\xi_1)}\left(\frac{uN_{\lfloor nt\rfloor}^2(x)}{a_n}\right)-1\right|
&\le& C_2\sum_{x\in\mathbb Z^d}\frac{|u|^{\beta'}N_{\lfloor nT\rfloor}^{2\beta'}(x)}{a_n^{\beta'}}
\end{eqnarray*}
which converges almost surely to 0 since, 
for every $\varepsilon>0$, the following
inequalities hold almost surely, for $n$ large enough
$$R_n\le n^{\frac 1{\alpha_0}+\varepsilon},\ \ 
  N_n^*\le n^{1-\frac 1{\alpha_0}+\varepsilon}\ \ \mbox{and}\ \ 
  a_n^{-1}\le n^{-2+\frac 2{\alpha_0}-\frac 2{\alpha_0\beta}
    +\varepsilon}$$
(see for example \cite{Spitzer,jainpruit83,BFFN}).
\end{proof}

%
%
%
%

\section{Examples}\label{examples}

The following examples are variants of Example 2.4 from \cite{DDMS}.
Observe that
$$\mathbb P(h(\xi_1,\xi_2)>z)=\int_E\mathbb P(h(x,\xi_2)>z)\, d\mathbb P_{\xi_1}(x)$$
and that
$$ \mathbb P(|h(\xi_1,\xi_2)|>z,|h(\xi_1,\xi_3)|>z')=
\int_E\mathbb P(|h(x,\xi_2)|>z)\mathbb P(|h(x,\xi_2)|>z')\, d\mathbb P_{\xi_1}(x).$$

\begin{itemize}
\item When $\beta<1$, one can take $E=\mathbb R^p$, the distribution of $\xi_1$
admitting a bounded density $f$ with respect to the Lebesgue measure on $E$
and $h(x,y)=\Vert x-y\Vert_\infty^{-p/\beta}{\mathbf 1}_{\{x\ne y\}}$.
This example fits Assumption \ref{HYP}.
Indeed, for every $z>0$, $\mathbb P(h(\xi_1,\xi_2)<-z)=0$ and
$$\mathbb P(h(x,\xi_2)>z)=\mathbb P(\Vert x-\xi_2\Vert_\infty<z^{-\frac\beta p})
    \sim_{z\rightarrow+\infty}2^pf(x)z^{-\beta}
\ \ \mbox{and}\ \ \mathbb P(h(x,\xi_2)>z)\le\Vert f\Vert_\infty 2^p z^{-\beta}.$$
So
$$\mathbb P(h(\xi_1,\xi_2)>z)\sim_{z\rightarrow+\infty}2^pz^{-\beta}
    \int_{\mathbb R^d}(f(x))^2\, dx $$
and
$$\mathbb P(|h(\xi_1,\xi_2)|>z,|h(\xi_1,\xi_3)|>z')\le (1+\Vert f\Vert_\infty 2^p)^2(\max(1,z)\max(1,z'))^{-\beta}.$$
\item Analogously, when $\beta\ge 1$, we can take
$E=\{\pm 1\}\times\mathbb R^p$, $h((\varepsilon,x),(\varepsilon',y))=\varepsilon\varepsilon'\Vert x-y\Vert_\infty^{-p/\beta}{\mathbf 1}_{\{x\ne y\}}$
and $\xi_1=(\varepsilon_1,\vec\xi_1)$ with $\varepsilon_1$ and $\vec\xi_1$ independent; $\varepsilon_1$
being centered and the distribution of $\vec\xi_1$ admitting a bounded
density $f$ with respect to the Lebesgue measure on $\mathbb R^p$.
Using the same argument
as for the previous example together with Remark \ref{RQE0}
we can verify that this example
satisfies Assumption \ref{HYP}.
\end{itemize}
Note that the case $ \beta=1 $ contains the more concrete kernel  $ h(x,y)=1/(x+y)$ for $x\ne y$ in association with some random variable 
$  \xi_1 $ having a bounded symmetric density on $ \mathbb{R} $.

\section{Convergence of finite distributions}\label{finitedistrib}

To simplify notations and the presentation of the proofs, we set
\begin{equation}
|z|_+^\beta:=|z|^\beta\ \ \mbox{and}\ \ |z|_-^\beta:=|z|^\beta\sgn(z)
\end{equation}
for any real number $z$.
Let $m\ge 1$ and $\theta_1,...,\theta_m\in\mathbb R$ and $0=t_0<t_1<...<t_m$. 

If $\alpha_0>1$, we will prove the convergence in distribution of the sequence of random variables
\begin{equation}\label{Folge1}
\left(a_n^{-1} \sum_{x,y\in\mathbb Z^{d_0}}\Big(\sum_{i=1}^m\theta_iN_{\lfloor nt_i\rfloor}(x)N_{\lfloor nt_i\rfloor}(y))h(\xi_x,\xi_y)\Big)\right)_{n\in\N}.
\end{equation}


If $\alpha_0= 1$, since the limit process will have independent increments, it will be more natural to prove the convergence in distribution of the sequence 
$$\left(a_n^{-1} \sum_{x,y\in\mathbb Z^{d_0}}\left(\sum_{i=1}^m\theta_i\Big(N_{\lfloor nt_i\rfloor}(x)N_{\lfloor nt_i\rfloor}(y)-N_{\lfloor nt_{i-1}\rfloor}(x)N_{\lfloor nt_{i-1}\rfloor}(y)\Big)\right)h(\xi_x,\xi_y)\right)_{n\in\N} .$$
Setting $d_{i,n}(x):=N_{\lfloor nt_i\rfloor}(x)-N_{\lfloor nt_{i-1}\rfloor}(x)$,
we observe that
\begin{equation}\label{F1}
\sum_{i=1}^m\theta_i\Big(N_{\lfloor nt_i\rfloor}(x)N_{\lfloor nt_i\rfloor}(y)-N_{\lfloor nt_{i-1}\rfloor}(x)N_{\lfloor nt_{i-1}\rfloor}(y)\Big)
=\sum_{i,j=1}^m \theta_{\max(i,j)}d_{i,n}(x)d_{j,n}(y)
\end{equation}
and hence, if $\alpha_0= 1$, it is sufficient  to  study  for fixed $\theta_{i,j}$ the sequence of random variables
\begin{equation} \label{Folge2}
 \left(a_n^{-1}\sum_{x,y\in\mathbb Z^d}\sum_{i,j=1}^m \theta_{i,j}d_{i,n}(x)d_{j,n}(y)h(\xi_x,\xi_y)\right)_{n\in\N} 
\end{equation}
(in view of applying the results to the particular case when $\theta_{i,j}=\theta_{\max(i,j)}$).


Therefore we have to prove the convergence in distribution of
$(a_n^{-1}\sum_{x,y\in\mathbb Z^{d_0}}\chi_{n,x,y}h(\xi_x,\xi_y))_n$,
with
$$\chi_{n,x,y}:= \sum_{i=1}^m\theta_iN_{\lfloor nt_i\rfloor}(x)N_{\lfloor nt_i\rfloor}(y)\ \ \mbox{if}\ \ \alpha_0>1$$
and
$$\chi_{n,x,y}:= \sum_{i,j=1}^m \theta_{i,j}d_{i,n}(x)d_{j,n}(y)\ \ \mbox{if}\ \ \alpha_0= 1.$$
The basic idea  is to identify the sequences in (\ref{Folge1}) and (\ref{Folge2}) as functionals of some sequence of suitably defined 
point processes and then to use Kallenberg theorem to prove convergence in law of those point processes. More precisely we will define in section
\ref{STEP2} the sequence of point processes on $ \mathbb R^*=\mathbb R\setminus\{0\} $ defined through
$$ {\mathcal N}_n(\tilde\omega,\xi):=\sum_{x,y\in \mathbb Z^{d_0}}\delta_{a_n^{-1}\zeta_{n,x,y}(\tilde\omega)h(\xi_{x},\xi_{y})} ,$$
where  $(\zeta_{n,x,y})_{n,x,y}$ are suitable random variables defined on some suitable probability space $(\tilde\Omega,\tilde{\mathcal F},\tilde{\mathbb P})$ such that, for every integer
$n$, the random variable $\sum_{x,y\in\mathbb Z^{d_0}}\zeta_{n,x,y}h(\xi_x,\xi_y)$ (with respect to $\mathbb P_{\xi}\otimes\tilde{\mathbb P}$) 
has the same law as $\sum_{x,y\in\mathbb Z^{d_0}}\chi_{n,x,y}h(\xi_x,\xi_y)$ (with respect to the original probability measure 
$\mathbb P$).

In section \ref{STEP1} we prove that the probability space $(\tilde\Omega,\tilde{\mathcal F},\tilde{\mathbb P})$ and the 
family $(\zeta_{n,x,y})_{n,x,y}$ can be chosen in such a way to satisfy
\begin{equation}\label{Konvergenz} \lim_{n\rightarrow +\infty}a_n^{-\beta}\sum_{x,y\in\mathbb Z^{d_0}}|\chi_{n,x,y}|_{\pm}^\beta=\tilde G^\pm\ \ \ \mbox{a.s.} ,\end{equation}
where $ \tilde{G} $ is a suitable random variable on $(\tilde\Omega,\tilde{\mathcal F},\tilde{\mathbb P})$. The construction will vary depending on whether
$ \alpha_0= 1 $ or $ \alpha_0>1 $.

The almost sure convergence in (\ref{Konvergenz}) will enable us to use Kallenberg theorem in section \ref{STEP2} to prove that for almost every 
$ \tilde\omega\in\tilde\Omega $ the sequence of 
point processes $ ({\mathcal N}_n(\tilde\omega,.))_{n\in\mathbb{N}} $ converges in law (with respect to $ \mathbb{P}_\xi $) toward a Poisson point process 
$ {\mathcal N}_{\tilde{\omega}} $ on $ \mathbb{R}^* $ with the following intensity function
$$z\mapsto \beta |z|^{-\beta-1}\frac{(c_0+c_1)\tilde G^+(\tilde\omega)+\sgn(z)(c_0-c_1)\tilde G^-(\tilde\omega)}2.$$

In section \ref{STEP3} we will see that $a_n^{-1}\sum_{x,y\in\mathbb Z^{d_0}}\zeta_{n,x,y}(\tilde\omega) h(\xi_x,\xi_y) $ equals $ 
\int_{\mathbb R^*}w\, \mathcal N_n(\tilde\omega,\xi,dw)$ which as $ n $ goes to infinity converges in distribution  toward  $ \int_{\mathbb R^*}w\, \mathcal N_{\tilde\omega}(dw)$.
We will also see in section \ref{STEP3} that this limit follows a stable law with characteristic function 
$\Phi_{(c_0+c_1)\tilde G^+(\tilde\omega),(c_0-c_1)\tilde G^-(\tilde\omega),\beta}$.
This will imply the convergence in distribution of the sequences in (\ref{Folge1}) and (\ref{Folge2})
toward the same stable limit.



%

%
%
\subsection{A result of convergence}\label{STEP1}

\subsubsection{Case $\alpha_0= 1$}

We define 
\begin{equation}
G_n^\pm:=a_n^{-\beta} \sum_{x,y\in\mathbb Z^{d_0}}
\left|\sum_{i,j=1}^m\theta_{i,j}d_{i,n}(x)d_{j,n}(y)\right|_\pm^\beta\ \ 
\mbox{and}\ \ 
G^\pm:= K_\beta^2\sum_{i,j=1}^m|\theta_{i,j}|_\pm^\beta(t_i-t_{i-1})(t_j-t_{j-1}),
\end{equation}
where $K_\beta$ is the constant defined in Theorems \ref{THM0} or \ref{THM00}
(depending on whether the random walk $(S_n)_n$ is transient or recurrent with $\alpha=d_0$).

\begin{lem}\label{LEM0}
If $\alpha_0= 1$, $(G_n^\pm)_n$ converges almost surely to $G^\pm$.
\end{lem}

Applying this lemma with $\theta_{i,j}=\theta_{\max(i,j)}$, we directly obtain the following almost sure equality
\begin{equation} \label{blaa}
\lim_{n\rightarrow\infty}a_n^{-\beta}\sum_{x,y\in\mathbb Z^{d_0}}\left|\sum_{i=1}^m\theta_i\Big(N_{\lfloor nt_i\rfloor}(x)N_{\lfloor nt_i\rfloor}(y)-N_{\lfloor nt_{i-1}\rfloor}(x)N_{\lfloor nt_{i-1}\rfloor}(y)\Big)\right|_\pm^\beta=K_\beta^2\sum_{j=1}^m|\theta_j|_\pm^\beta(t_j^2-t_{j-1}^2) .
\end{equation}

\begin{proof}[Proof of Lemma \ref{LEM0}]
We proceed as in \cite{Cerny,cast}.
\begin{itemize}
\item Let $k$ be a nonnegative integer. Let us prove that
\begin{equation}\label{cvgce}
\lim_{n\rightarrow +\infty}
(b_{n,k})^{-2}\sum_{x,y\in\mathbb Z^{d_0}}\left(\sum_{i,j= 1}^m \theta_{i,j}d_{i,n}(x)d_{j,n}
(y)\right)^k=(K_k)^2\sum_{i,j=1}^m(\theta_{i,j})^k(t_i-t_{i-1})(t_j-t_{j-1})\ a.s.,
\end{equation}
with $b_{n,k}:=n(\log n)^{k-1}$ if $(S_n)_n$ is recurrent (and $\alpha=d_0$) and with
$b_{n,k}:=n$ if $(S_n)_n$ is transient (extending
the definition of $K_\beta$ given in Theorems \ref{THM0}
or \ref{THM00} to any nonnegative real number $\beta$).
Due to \cite[p. 10]{kest} (transient case) and to \cite{Cerny} (null recurrent case), we know that
\begin{equation}\label{EQ}
\forall i\in\{1,...,m\},\ \ \lim_{n\rightarrow\infty}
  (b_{n,k})^{-1}\sum_{x\in\mathbb Z^{d_0}}(d_{i,n}(x))^k=K_k(t_i-t_{i-1})\ \ a.s..
\end{equation}
Following some argument from \cite{cast}, we observe that
$$\left|\sum_{x,y\in\mathbb Z^{d_0}}\left(\sum_{i,j=1}^m\theta_{i,j}d_{i,n}(x)d_{j,n}(y)\right)^k -\sum_{x,y\in\mathbb Z^{d_0}}\sum_{i,j=1}^m(\theta_{i,j})^k
   (d_{i,n}(x)d_{j,n}(y))^k\right|$$
\begin{eqnarray*}
&\le& \max_{i,j}|\theta_{i,j}|^k\sum_{((i_1,j_1),...,(i_k,j_k))\in\mathcal I}
\sum_{x,y\in\mathbb Z^{d_0}}\prod_{\ell=1}^kd_{i_\ell,n}(x)d_{j_\ell,n}(y)\\
&\le&  \max_{i,j}|\theta_{i,j}|^k\left(\sum_{x,y\in\mathbb Z^{d_0}}\left(\sum_{i,j=1}^md_{i,n}(x)d_{j,n}(y)\right)^k -\sum_{x,y\in\mathbb Z^{d_0}}\sum_{i,j=1}^m
   (d_{i,n}(x)d_{j,n}(y))^k\right)\\
&\le&\max_{i,j}|\theta_{i,j}|^k
    \left(\left(\sum_{x\in\mathbb Z^{d_0}}(N_{\lfloor nt_m\rfloor}(x))
^k\right)^2-\left(\sum_{i=1}^m\sum_{x\in\mathbb Z^{d_0}}(d_{i,n}(x))^k\right)^2\right),
\end{eqnarray*}
where $\mathcal I$ denotes the set of $((i_1,j_1),...,(i_k,j_k))\in
(\{1,...,m\}^2)^k$ such that $\#\{(i_1,j_1),...,(i_k,j_k)\}\ge 2$.
Due to (\ref{EQ}), we conclude that
this term is in $o((b_{n,k})^2)$. 
\item Assume here that $(S_n)_n$ is recurrent and $\alpha=
d_0$. Let us define
$$W_n:=\frac {(c_3)^2}{\log^2 n}\sum_{i,j=1}^m\theta_{i,j}d_{i,n}(V_n)d_{j,n}(V_n'), $$
with $(V_n,V_n')$ such that the conditional distribution of $(V_n,V_n')$ given $S$
is the uniform distribution on the set $\{z:N_{\lfloor nt_m\rfloor}(z)\ge 1\}^2$. We observe that
\begin{equation}\label{EWn1}
{\mathbb E}[|W_n|_\pm^u|S]=\frac{c_3^{2u}}{\log^{2u} n}
    \frac 1{R_{\lfloor nt_m\rfloor}^2}\sum_{x,y\in\mathbb Z^{d_0}}
    \left|\sum_{i,j=1}^m\theta_{i,j}d_{i,n}(x)d_{j,n}(y)\right|_\pm^u
\end{equation}
for all $ u>0 $.
Recall that $R_{\lfloor nt_m\rfloor}$ is the cardinal of $\{z:N_{\lfloor nt_m\rfloor}(z)\ge 1\}$ and that $R_n\sim c_3 n/\log n$ a.s..
Due to (\ref{cvgce}) and since $K_k=\Gamma(k+1)/c_3^{k-1}$, we conclude that, for every non negative integer $k$, we have, almost surely,
$$ \lim_{n\rightarrow +\infty}{\mathbb E}[(W_n)^k|S]
    =  (\Gamma(k+1))^2
\sum_{i,j=1}^m(\theta_{i,j})^k\frac{t_i-t_{i-1}}{t_m}\frac{t_j-t_{j-1}}{t_m}
=\mathbb E[W_\infty^k],$$
with $W_\infty=\theta_{V,V'}TT'$ where $ V',V,T,T'$ are independent random variables, $T$ and $T'$ having exponential distribution of parameter 1,
$V$ and $V'$ being such that
$\mathbb P(V=i)=\mathbb P(V'=i)=\frac {t_i-t_{i-1}}{t_m}$
for every $i\in\{1,\dots,m\}$.
From which we conclude that, almost surely, $(W_n|S)_n$ converges
in distribution to $W_\infty$ and that
\begin{equation}\label{EWn2}
\lim_{n\rightarrow +\infty}
{\mathbb E}[|W_n|_\pm^\beta|S]={\mathbb E}[|W_\infty|_\pm^\beta]\ \ a.s..
\end{equation}
The proof now follows due to (\ref{EWn1}) and (\ref{EWn2}).
\item Assume now that $(S_n)_n$ is transient and set this time
$$W_n:=\sum_{i,j=1}^m\theta_{i,j}d_{i,n}(V_n)d_{j,n}(V_n'), $$
for the same choice of $(V_n,V_n')$ as in the previous case.
Observe that
$${\mathbb E}[|W_n|_\pm^u|S]=
    \frac 1{R_{\lfloor nt_m\rfloor}^2}\sum_{x,y\in\mathbb Z^{d_0}}
    \left|\sum_{i,j=1}^m\theta_{i,j}d_{i,n}(x)d_{j,n}(y)\right|_\pm^u $$
for all $ u>0 $.
We recall now, that $R_n\sim pn$ with $p:=\mathbb P(S_k\ne 0,\ \forall k\ge 1)
=2/(\mathbb E[N_\infty]+1)$ (see \cite[p. 35]{Spitzer}).
Due to (\ref{cvgce}) and since $ K_k=\mathbb E[N_\infty^{k-1}] $, we obtain that, for every nonnegative integer $k$, we have
almost surely
$$\lim_{n\rightarrow +\infty}{\mathbb E}[W_n^k|S]=\left(\frac{\mathbb E[N_\infty^{k-1}]}{p}\right)^2\sum_{i,j=1}^m(\theta_{i,j})^k\frac{t_i-t_{i-1}}{t_m}
   \frac{t_j-t_{j-1}}{t_m}  .$$
So $(W_n|S)_n$ converges in distribution to $TT'\theta_{V,V'}$ where $V,V',T,T'$
are independent random variables such that
$$\forall i\in\{1,...,m\},\ \ \mathbb P(V=i)=\mathbb P(V'=i)=\frac{t_i-t_{i-1}}
{t_m} $$
and
$$\forall m\ge 1,\ \ \mathbb P(T=m)=\mathbb P(T'=m)=\frac{\mathbb P(N_\infty=m)}{mp}=(1-p)^{m-1}p.$$
Indeed, setting $N_\infty(0):=\sup_nN_n(0)$, we have $\mathbb P(N_\infty(0)=k)=(1-p)^{k}p$ 
for every integer $k\ge 0$. 
Note that $ N_\infty= 1+N_\infty(0)+\tilde{N}_\infty(0) $ where $ \tilde{N}_\infty(0)=\sum_{n\leq -1}{\bf 1}_{\{S_n=0\}} $ which is an independent copy of 
$ N_\infty(0) $.  Hence we have 
$$\mathbb P(N_\infty=m)=\sum_{k,\ell\ge 0:k+\ell=m-1}
\mathbb P(N_\infty(0)=k)\mathbb P(N_\infty(0)=\ell)=mp^2(1-p)^{m-1},$$
for every integer $m\ge 1$.
Therefore
$$\lim_{n\rightarrow +\infty}{\mathbb E}[|W_n|_\pm^\beta|S]=\left(\frac{\mathbb E[N_\infty^{\beta-1}]}{p}\right)^2\sum_{i,j=1}^m|\theta_{i,j}|_\pm^\beta\frac{t_i-t_{i-1}}{t_m}
   \frac{t_j-t_{j-1}}{t_m} \ \  a.s..$$
This finishes the proof in this case.
\end{itemize}
\end{proof}

Since in the main proof we want to treat simultaneously the cases $ \alpha_0= 1 $ and $ \alpha_0>1 $, we have to introduce some 
additional notations  which will have its counterpart in the case $ \alpha_0>1 $. 
So for $\alpha_0= 1$, we set $\tilde N_{n,t_i}(x):=N_{\lfloor nt_i\rfloor}(x)$, $\tilde N_n^*:=N_{\lfloor nt_m\rfloor}^*$, 
$\tilde R_n:=R_{\lfloor nt_m\rfloor}$,
$\tilde G_n^\pm:=G_n^\pm$ and $\tilde G^\pm:=G^\pm$.
We fix $\varepsilon>0$ such that $\varepsilon<1/(3+4\beta)$ and
$(3+4\gamma)\varepsilon<\frac{4\gamma}\beta-3$. 
If $\beta<4/3$, we assume moreover that $3-\frac{4\min(1,\gamma)}\beta+7\varepsilon<0$ (with $\gamma$ of Item (iv) of Assumption \ref{HYP}).
If $\beta\ge 4/3$, we assume that $3-\frac{4(\theta'+1)}{\beta}+(4\theta'+7)\varepsilon<0$ (with $\theta'$ of Item (vi) of Assumption \ref{HYP}). 
We write $\tilde{\mathcal F}$ for the sub-algebra generated by
$S$. We consider the set
$\tilde \Omega_0\in\tilde{\mathcal F}$ on which $(G_n^+,G_n^-,n^{-\varepsilon}N_n^*)$ converges
to $(G^+,G^-,0)$.
When $\alpha_0= 1$, we will make no distinction between $\mathbb E$ and $\mathbf E$ nor
between $\mathbb P$ and $\mathbf P$.

\subsubsection{Case $\alpha_0>1$}
For every $b,t\ge 0$, we set 
$$F_{n,t}(b):={n^{-1}}\int_0^{n^{\frac 1\alpha}b}N_{\lfloor nt\rfloor}(\lfloor y\rfloor)\, dy,\ \ 
F_{n,t}(-b):=-{n^{-1}}\int_{-n^{\frac 1\alpha}b}^0N_{\lfloor nt\rfloor}(\lfloor y\rfloor)\, dy,$$
$$F_t(b)=\int_0^b\mathcal L_t(x)\, dx \ \ \mbox{and}\ \ F_t(-b)=-\int_{-b}^0\mathcal L_t(x)\, dx,$$
(recall that $\mathcal L_s(x)$ is the local time of $(Y_t)_t$ at position $x$ and up to time $s$).
It was proved in \cite{kest} that $ F_{n,t}(b) $ converges towards $ F_t(b) $ in distribution. We prove some vector version of this result. Let us define 
$$G_n^\pm:=a_n^{-\beta}\sum_{x,y\in\mathbb Z}\left|\sum_{i=1}^m\theta_iN_{\lfloor nt_i\rfloor}(x)N_{\lfloor nt_i\rfloor}(y)\right|_\pm^\beta\ \ 
\mbox{and}\ \  
G^\pm:=\int_{\mathbb R^2}
\left|\sum_{i=1}^m\theta_i\mathcal L_{t_i}(x)\mathcal L_{t_i}(y)\right|_\pm^\beta\, dxdy.$$

\begin{lem}\label{LEM1}
The finite distributions of $(F_{n,t_1},\cdots,F_{n,t_m},G_n^+,G_n^-)_n$ 
converge to the finite distributions of $(F_{t_1},\cdots,F_{t_m},G^+,G^-)$, i.e.
$((F_{n,t_i}(b_j))_{i=1,\cdots,m,j=1,\dots,q},G_n^+,G_n^-)_n$ converges in distribution to the random variable
$((F_{t_i}(b_j))_{i=1,\cdots,m,j=1,\dots,q},G^+,G^-)_n$,
for every integer $q\ge 1$ and every real numbers $b_1,...,b_q$.
\end{lem}

\begin{proof}
The proof of this convergence result follows mainly the proof of Lemma 6 of \cite{kest}. 
For any real number $\tau>0$ and any positive integers $n$ and $M$, we define
$$V^\pm(\tau,M,n):=\tau^{2-2\beta}\sum_{|k|,|\ell|\le M}|T(k,\ell,n)|^\beta_\pm,$$
where
$$T(k,\ell,n):=n^{-2}\sum_{j=1}^m\theta_j\sum_{x=\lceil k\tau n^{\frac 1\alpha}\rceil}^{\lceil (k+1)\tau n^{\frac 1\alpha}\rceil-1}
\sum_{y=\lceil \ell\tau n^{\frac 1\alpha}\rceil}^{\lceil (\ell+1)\tau n^{\frac 1\alpha}\rceil-1} N_{\lfloor nt_j\rfloor}(x)N_{\lfloor nt_j\rfloor}(y).$$
As in \cite{kest}, we decompose $G_n^\pm-V^\pm(\tau,M,n)$ as follows
$$G_n^\pm-V^\pm(\tau,M,n)=U^\pm(\tau,M,n)+W_1^\pm(\tau,M,n)+W_2^\pm(\tau,M,n),$$
with
$$U^\pm(\tau,M,n):=n^{-2\delta\beta}\sum_{(x,y)\in A_{\tau,M,n}}\left|
     \sum_{j=1}^m\theta_jN_{\lfloor nt_j\rfloor}(x)N_{\lfloor nt_j\rfloor}(y)\right|_\pm^\beta, $$
where
$A_{\tau,M,n}:=\mathbb Z^2\setminus\{\lceil -M\tau n^{\frac 1\alpha}\rceil,...,
    \lceil (M+1)\tau n^{\frac 1\alpha}\rceil-1\}^2$,
$$W_1^\pm(\tau,M,n):=\sum_{|k|,|\ell|\le M}\sum_{(x,y)\in E_{k,n}\times E_{\ell,n}}n^{-2\delta\beta}W_{1,k,\ell}^\pm(x,y), $$
where $E_{k,n}:=\{\lceil k\tau n^{\frac 1\alpha}\rceil,...,
    \lceil (k+1)\tau n^{\frac 1\alpha}\rceil-1\}$,
$$ W_{1,k,\ell}^\pm(x,y):=\left|
     \sum_{j=1}^m\theta_jN_{\lfloor nt_j\rfloor}(x)N_{\lfloor nt_j\rfloor}(y)\right|_\pm^\beta- n^{2\beta}(\#E_{k,n}\, \#E_{\ell,n})^{-\beta}|T(k,\ell,n)|_\pm^\beta$$
and
$$W_2^\pm(\tau,M,n):= \sum_{|k|,|\ell|\le M}\left\{n^{2\beta-2\delta\beta}
(\# E_{k,n}\, \# E_{\ell,n})^{1-\beta}-\tau^{2-2\beta}\right\}|T(k,\ell,n)|_\pm^\beta.$$
The proof follows now in five steps:\\[1ex]
1) Observe that, due to \cite[Lemma 1]{kest}, there exists a function $ \eta $ satisfying $\lim_{x\rightarrow+\infty}\eta(x)=0$ such that
\begin{eqnarray} \label{KSL1}
\sup_n\mathbb P\Big(U^\pm(\tau,M,n)\ne 0\Big)\le\sup_n
  \mathbb P\Big(\exists x\ :\ |x|\ge M\tau n^{\frac 1\alpha}\ \mbox{and}\ 
  N_{\lfloor nt_m\rfloor}(x)\ne 0\Big)=\eta(M\tau).
\end{eqnarray}
2)  We prove that there exists some $ K>0 $ and $ u>0 $ such that for all $ M>1 $ one has
\begin{eqnarray} \label{KSL2}
      \sup_n \mathbb E[|W_1^\pm(\tau,M,n)|]\leq  K(M\tau)^2\tau^u.\end{eqnarray}
We first do the case $\beta\le 1$. Using the fact that $||a|_\pm^\beta-|b|_\pm^\beta|\le 2^{1-\beta}|a-b|^\beta$, we have
\begin{eqnarray*}
&& 2^{\beta-1}{\mathbb E}[|W_{1,k,\ell}^\pm(x,y)|] \\
&\le& \mathbb E\left[\left|
     \sum_{j=1}^m\theta_jN_{\lfloor nt_j\rfloor}(x)N_{\lfloor nt_j\rfloor}(y)- n^{2}(\#E_{k,n}\, \#E_{\ell,n})^{-1}T(k,\ell,n)\right|^\beta\right]\\
&\le& \left\Vert
     \sum_{j=1}^m\theta_jN_{\lfloor nt_j\rfloor}(x)N_{\lfloor nt_j\rfloor}(y)- n^{2}(\#E_{k,n}\, \#E_{\ell,n})^{-1}T(k,\ell,n)\right\Vert_2^\beta\\
&\le&(\#E_{k,n}\, \#E_{\ell,n})^{-\beta} \left\Vert
     \sum_{j=1}^m\sum_{(x',y')\in E_{k,n}\times E_{\ell,n}}
    \theta_j\left(N_{\lfloor nt_j\rfloor}(x)N_{\lfloor nt_j\rfloor}(y)- N_{\lfloor nt_j\rfloor}(x')N_{\lfloor nt_j\rfloor}(y')\right)\right\Vert_2^\beta\\
&\le&(\#E_{k,n}\, \#E_{\ell,n})^{-\frac\beta 2} 
\left(\sum_{i=1}^m\theta_i^2\sum_{j=1}^m\sum_{(x',y')\in E_{k,n}\times E_{\ell,n}}
\left\Vert
   \left(N_{\lfloor nt_j\rfloor}(x)N_{\lfloor nt_j\rfloor}(y)- N_{\lfloor nt_j\rfloor}(x')N_{\lfloor nt_j\rfloor}(y')\right)\right\Vert_2^2\right)^{\frac\beta 2},
\end{eqnarray*}
due to the Cauchy-Schwarz inequality.
Now we have to estimate
$$\sum_{(x',y')\in E_{k,n}\times E_{\ell,n}}
{\mathbb E}\Big[|N_{\lfloor nt_j \rfloor}(x)N_{\lfloor nt_j \rfloor}(y)-N_{\lfloor nt_j \rfloor}(x')N_{\lfloor nt_j \rfloor}(y')|^2\Big],$$
for $(x,y)\in E_{k,n}\times E_{\ell,n}$.
To this end, we use $\mathbb E[|ab-a'b'|^2]\le 2 \Vert a\Vert_4^{2}
  \Vert b-b'\Vert_4^2+\Vert a-a'\Vert_4^{2} \Vert b'\Vert_4^2$ together
with the fact that
$$\sup_x\mathbb E[(N_n(x))^4]=O(n^{4-\frac 4\alpha})\ \ \mbox{and}\ \ \sup_{y\ne z}\frac{\mathbb E[|N_n(y)-N_n(z)|^4]}{|y-z|^{2\alpha-2}}=
O(n^{2-\frac 2\alpha})
$$
(see for example \cite[p.77]{jainpruit83} for the last estimate).
This gives, 
\begin{equation}\label{erreurproduit}
{\mathbb E}[|N_{\lfloor nt_j \rfloor}(x)N_{\lfloor nt_j \rfloor}(y)-N_{\lfloor nt_j \rfloor}(x')N_{\lfloor nt_j \rfloor}(y')|^2]
\le C \tau^{\alpha-1} n^{4-\frac 4\alpha},
\end{equation}
for every $(x,y),(x',y')\in E_{k,n}\times E_{\ell,n}$
and for some $C>0$ independent of $(\tau,M,n,k,\ell)$.
Therefore, we obtain
$${\mathbb E}[|W_1^\pm(\tau,M,n)|]\le C'(2M+1)^2\tau^{2+\frac\beta 2(\alpha-1)},$$
where $C'$ does not depend on $(\tau,M,n)$. From this we conclude in the case $\beta\le 1$.\\[1ex]
When $\beta>1$, we use $||a|_\pm^\beta-|b|_\pm^\beta|\le\beta|a-b|(|a|^{\beta-1}+|b|^{\beta-1})$ combined with the Cauchy-Schwarz inequality and obtain
\begin{eqnarray*}
 &&{\mathbb E}[|W_{1,k,\ell}^\pm(x,y)|]\\
&\le& \beta
(\#E_{k,n}\, \#E_{\ell,n})^{-1}
\left\Vert 
     \sum_{j=1}^m\theta_j\sum_{(x',y')\in E_{k,n}\times E_{\ell,n}}(N_{\lfloor nt_j\rfloor}(x)N_{\lfloor nt_j\rfloor}(y)-N_{\lfloor nt_j\rfloor}(x')N_{\lfloor nt_j\rfloor}(y')) \right\Vert_2\times\\
&\ &\ \ \ \ 
\displaystyle \times 
\left\Vert \left |\sum_{j=1}^m\theta_j N_{\lfloor nt_j\rfloor}(x)
  N_{\lfloor nt_j\rfloor}(y) \right|^{\beta-1}+
     \Big(n^{2}(\#E_{k,n}\, \#E_{\ell,n})^{-1}|T(k,\ell,n)|\Big)^{\beta-1} \right\Vert_2\\
&\le& \beta
(\#E_{k,n}\, \#E_{\ell,n})^{-1}\sum_{j=1}^m|\theta_j|\sum_{(x',y')\in E_{k,n}\times E_{\ell,n}}
\left\Vert 
     (N_{\lfloor nt_j\rfloor}(x)N_{\lfloor nt_j\rfloor}(y)-N_{\lfloor nt_j\rfloor}(x')N_{\lfloor nt_j\rfloor}(y')) \right\Vert_2\times\\
&\ &\ \ \ \times 
\left(\left\Vert \sum_{j=1}^m\theta_j N_{\lfloor nt_j\rfloor}(x)
  N_{\lfloor nt_j\rfloor}(y) \right\Vert
_{2(\beta-1)}^{\beta-1}+n^{2\beta-2}(\#E_{k,n}\, \#E_{\ell,n})^{1-\beta}\left\Vert
     T(k,\ell,n) \right\Vert
_{2(\beta-1)}^{\beta-1}\right)\\
&\le& C(\tau^{\alpha-1}n^{4-\frac 4\alpha})^{\frac 12}\left(\sup_{x'}\Vert N_{\lfloor nt_m\rfloor}(x')\Vert_{4(\beta-1)}^{2\beta-2}
   + \right.\\
&\ &\ \ \ \ \left.+n^{2\beta-2}(\tau n^{\frac 1\alpha})^{2-2\beta}
    \left(n^{-2}(\tau n^{\frac 1\alpha})^2\sup_{x'}\Vert N_{\lfloor nt_m\rfloor}  
    (x')\Vert_{4(\beta-1)}^2\right)^{\beta-1}\right),
\end{eqnarray*}
due to the Cauchy Schwarz inequality and to (\ref{erreurproduit}).
Hence we have
$${\mathbb E}[|W_{1,k,\ell}^\pm(x,y)|]\le C'\tau^{\frac{\alpha-1}2}n^{2-\frac 2\alpha} n^{\left(1-\frac 1\alpha\right)
2(\beta-1)}= C'\tau^{\frac{\alpha-1}2}n^{2\beta\left(1-\frac 1\alpha\right)}$$
for some $C'>0$ and so
$${\mathbb E}[|W_1^\pm(\tau,M,n)|]
\le C''(2M+1)^2\tau^{2+\frac{\alpha-1}2},$$
where $C''$ does not depend on $(\tau,M,n)$ and we conclude in the case when
$\beta>1$.\\[1ex]
3)  We notice that 
$$ \mathcal V^\pm(\tau,M):=\tau^{2-2\beta}\sum_{|k|,|\ell|\le M}
\left|\int_{k\tau}^{(k+1)\tau}\int_{\ell\tau}^{(\ell+1)\tau}
   \sum_{j=1}^m\theta_j\mathcal L_{t_j}(x)\mathcal L_{t_j}(y)\, dxdy\right|_\pm^\beta $$ 
converge to
$G^\pm$ as $(\tau,M\tau)\rightarrow(0,\infty)$, since the local times $x\mapsto\mathcal L_{t_j}(x)$
are almost surely continuous and compactly supported (see \cite{kest}).\\[1ex]
4)  We observe that, for every choice of $(\tau,M)$ the sequence $(W_2^\pm(\tau,M,n))_n$
converges in probability to 0 as $ n\rightarrow\infty$. This comes from the fact that for every $(k,\ell)$ the sequence $(T(k,\ell,n))_n$
converges in distribution to $$\sum_{j=1}^m\theta_j\int_{k\tau}^{(k+1)\tau}
\mathcal L_{t_j}(x)\, dx\int_{\ell\tau}^{(\ell+1)\tau}
\mathcal L_{t_j}(y)\, dy$$ and the fact that the sequence $(n^{2\beta(1-\delta)}
(\#E_{k,n}\, \#E_{\ell,n})^{1-\beta}-\tau^{2-2\beta})_n$ converges to 0.\\[1ex]
5) For every choice of $(\tau,M)$, for every $q$ and every real numbers $b_1,...,b_q$, the sequence of random variables 
$((F_{n,t_i}(b_j))_{i,j},V^+(\tau,M,n),V^-(\tau,M,n)))_n$ converges
in distribution to $((F_{t_i}(b_j))_{i,j},\mathcal V^+(\tau,M),\mathcal V^-(\tau,M))$. Indeed, we recall that
$$V^\pm(\tau,M,n):=\tau^{2-2\beta}\sum_{|k|,|\ell|\le M}|T(k,\ell,n)|^\beta_\pm$$ and notice that
$$T(k,\ell,n)=\left(\sum_{j=1}^m\theta_j\Big(F_{n,t_j}((k+1)\tau)-F_{n,t_j}(k\tau)\Big)
\Big(F_{n,t_j}((\ell+1)\tau)-F_{n,t_j}(\ell\tau)\Big)\right) +O(n^{-1})$$
6) Now we conclude. Let $ z_{i,j},z_\pm\in\mathbb{R} $ and 
$ \epsilon>0 $.
Due to Points 1, 2 and 3, we fix $ M>1 $ and $\tau>0 $ such, for every $n$, we have
\begin{equation}\label{numero1}
\mathbb E\left[\left|e^{i(z_+G_n^++z_-G_n^-)}
     - e^{i(z_+(V^+(\tau,M,n)+W_2^+(\tau,M,n))+
          z_-(V^-(\tau,M,n)+W_2^-(\tau,M,n)))} \right|\right]<\varepsilon 
\end{equation}
and
\begin{equation}\label{numero2}
\mathbb{E}\left[\left|e^{i\left(z_+\mathcal V^+(\tau,M)+z_-\mathcal V^-(\tau,M)\right)}- 
            e^{i\left(z_+ G^++z_- G^-\right)}\right|\right] <\epsilon  . 
\end{equation}
Due to Points 4 and 5 for this choice of $(M,\tau)$, there
exists $n_0$ such that for every $n\ge n_0$, 
\begin{equation}\label{numero3}
    \mathbb{E}\left[\left|e^{iz_+W_2^+(\tau,M,n)+iz_-W_2^-(\tau,M,n)}-1\right|\right]  <\epsilon
\end{equation}
and 
\begin{equation}\label{numero4}
 \Bigg| \mathbb{E}\left[e^{i\left(\sum_{i,j} z_{ij}F_{t_i}(b_j))+z_+\mathcal V^+(\tau,M)+z_-\mathcal V^-(\tau,M)\right)}\right]   -  \mathbb{E}\left[e^{i\left(\sum_{i,j} z_{ij}F_{n,t_i}(b_j))+z_+V^+(\tau,M,n)+z_- V^-(\tau,M,n)\right)}\right]\Bigg| <\epsilon  .
\end{equation}
Hence, for every $n\ge n_0$, we have
\begin{eqnarray*}
 && \Bigg| \mathbb{E}\left[e^{i\left(\sum_{i,j} z_{ij}F_{t_i}(b_j))+z_+G^++z_- G^-\right)}\right]
 - \mathbb{E}\left[e^{i\left(\sum_{i,j} z_{ij}F_{n,t_i}(b_j))+z_+G_n^++z_-G_n^-\right)}\right]  \Bigg| \\
 &\leq&  3\epsilon+ 
\Bigg| \mathbb{E}\left[e^{i\left(\sum_{i,j} z_{ij}F_{t_i}(b_j))+z_+\mathcal V^+(\tau,M)+z_-\mathcal V^-(\tau,M)\right)}\right]
 - \mathbb{E}\left[e^{i\left(\sum_{i,j} z_{ij}F_{n,t_i}(b_j))+z_+V^+(\tau,M,n)+z_-V^-(\tau,M,n)\right)}\right]  \Bigg| \\
&\leq& 4\varepsilon.
\end{eqnarray*}
where we used \eqref{numero1}, \eqref{numero2}, \eqref{numero3}
for the first inequality and \eqref{numero4} for the last one.
\end{proof}

Let $ {\mathcal  C}$ be the set of continuous functions $g:\mathbb R\rightarrow[-t_m,t_m]$.
We endow this set with the following metric $D$ corresponding to the uniform convergence on every
compact:
$$D(g,h):=\sum_{N\ge 1}2^{-N}\sup_{x\in[-N;N]}|g(x)-h(x)|.$$

\begin{lem}\label{LEM1bis}
The sequence $ (F_{n,t_1},...,F_{n,t_m})_{n\in\mathbb{N}} $ is tight in $ ({\mathcal  C},D)^m $.
\end{lem}

\begin{proof}
It is enough to prove the tightness of $ F_{n,t_i} $ for all $ i\in\{1,...,m\} $. To simplify notations in this proof we use $ F_n $ to denote $ F_{n,t_i}/t_i $ and 
$ F $ to denote $ F_{t_i}/t_i $.
As usual, for any $f\in {\mathcal  C} $, we denote by $\omega(f,\cdot)$ the modulus of continuity of $f$. 
Since $F_n(0)=0$ for every $n$, it is enough
to prove
\begin{equation}
\forall\varepsilon>0,\ \ \ 
\lim_{\delta\rightarrow 0}\limsup_{n\rightarrow +\infty}\mathbb P(\omega(F_n,\delta)\ge\varepsilon)=0
\end{equation}
(see \cite[p.83]{Billingsley}).
Let $\varepsilon>0$ and $\varepsilon_0>0$.
Let $M>0$ be such that $\mathbb P\Big(|F(M)-F(-M)|\le 1-(\varepsilon/2)\Big)\le\varepsilon_0/2$. Since $(F_n(M)-F_n(-M))_n$ converges in distribution
to $F(M)-F(-M)$, we have
\begin{equation}\label{bounded}
\limsup_{n\rightarrow+\infty}\mathbb P\Big(|F_n(M)-F_n(-M)|\le 1-(\varepsilon/2)\Big) \le \mathbb P\Big(|F(M)-F(-M)|\le 1-(\varepsilon/2)\Big)\le\varepsilon_0/2.
\end{equation}
Let $\delta_0>0$ be such that, for every $\delta\in(0,\delta_0)$,
$\mathbb P(\omega(F,\delta)\ge\varepsilon/2)\le\varepsilon_0/2$
(since $F$ is almost surely uniformly continuous). Since the finite distributions of $(F_n)_n$ 
converge to the finite distribution of $F$, we have
\begin{eqnarray}\label{B1}
 && \limsup_{n\rightarrow+\infty}\mathbb P\left(\exists k=-\left\lceil\frac M\delta\right\rceil,...,\left\lceil\frac M\delta\right\rceil,\ 
\left|F_n\left(k\delta\right)-F_n\left((k+1)\delta\right)\right|\ge\frac{\varepsilon}2\right) \\
&\le&\mathbb P\left(\exists k=-\left\lceil\frac M\delta\right\rceil,...,\left\lceil\frac M\delta\right\rceil,\ 
\left|F\left(k\delta\right)-F\left((k+1)\delta\right)\right|\ge\frac{\varepsilon}2\right)\nonumber\\
&\le& \mathbb P\left(\omega(F,\delta)\ge\frac\varepsilon2\right)\le\frac{\varepsilon_0}2\label{bounded2}\nonumber.
\end{eqnarray}
Putting (\ref{bounded}) and (\ref{B1}) together, we obtain
that, for every $\delta<\delta_0$, we have
\bigskip

\begin{eqnarray*}
\limsup_{n\rightarrow +\infty}\mathbb P(\omega(F_n,\delta)\ge\varepsilon)
   &\leq&   \limsup_{n\rightarrow+\infty}\mathbb P\left(\exists k=-\left\lceil\frac M\delta\right\rceil,...,\left\lceil\frac M\delta\right\rceil,\ 
\left|F_n\left(k\delta\right)-F_n\left((k+1)\delta\right)\right|\ge\frac{\varepsilon}2\right) \\
 &&+ \limsup_{n\rightarrow+\infty}  \mathbb P\Big(|F_n(M)-F_n(-M)|\le 1-(\varepsilon/2)\Big) 
\end{eqnarray*}
and so
$$\limsup_{n\rightarrow +\infty}\mathbb P(\omega(F_n,\delta)\ge\varepsilon)
\le\varepsilon_0.$$
\end{proof}

Due to Lemma \ref{LEM1} and Lemma \ref{LEM1bis}, the sequence 
$(F_{n,t_1},\dots,F_{n,t_m},G_n^+,G_n^-)_n$ 
converges in distribution to $(F_{t_1},\dots,F_{t_m},G^+,G^-)$ in $({\mathcal  C},D)^m\times(\mathbb R,|\cdot|)^2$.\\[1mm]
We fix $\varepsilon\in(0,\beta\delta/(1+\beta))$ such that
$(3+4\beta)\varepsilon<1/\alpha$ and $(3+4\gamma)\varepsilon\alpha<\frac{4\gamma}\beta-3$ (this is possible due to $\gamma>3\beta/4$). 
If $\beta<4/3$, we assume moreover that $\frac 3\alpha-\frac{4\min(1,\gamma)}{\alpha\beta}+7\varepsilon<0$ (with $\gamma$ of Item (iv) of 
Assumption \ref{HYP}).
If $\beta\ge 4/3$, we assume also that $\frac 1\alpha\left(
3-\frac{4(\theta'+1)}\beta\right)+(4\theta'+7)\varepsilon<0$
(with $\theta'$ of Item (vi) of Assumption \ref{HYP}).
Using for example \cite{jainpruit83} for the maximal occupation time
and appendix of \cite{BFFN} for the range, we know that
$(n^{-1/\alpha-\varepsilon}R_n,n^{(1/\alpha)-1-\varepsilon}N_n^*)_n$ converges almost surely to 0.
Therefore the sequence $(F_{n,t_1},\dots,F_{n,t_m},G_n^+,G_n^-,n^{-1/\alpha-\varepsilon}R_n,n^{(1/\alpha)-1-\varepsilon}N_n^*)_n$
converges in distribution to $(F_{t_1},\dots,F_{t_m},G^+,G^-,0,0)$ in $({\mathcal  C},D)^m\times(\mathbb R,|\cdot|)^4$.

Now using the Skorokhod representation theorem (see \cite{Dudley} p.1569) (since $({\mathcal  C},D)$ and $\mathbb R$ are 
separable and complete), we know that there exists a  probability space
$(\tilde\Omega,\tilde{\mathcal F},\tilde{\mathbb P})$ with random variables 
$$(\tilde F_{n,t_1},\dots,\tilde F_{n,t_m},\tilde G_n^+,\tilde G_n^-,\tilde R_{n},\tilde N_n^*)_n\ \mbox{and}\ (\tilde F_{t_1},\dots,\tilde F_{t_m},\tilde G^+,\tilde G^-)$$
 defined on $(\tilde\Omega,\tilde{\mathcal F},\tilde{\mathbb P})$ such that
\begin{itemize} 
\item for every integer $n$,
$(\tilde F_{n,t_1},\dots,\tilde F_{n,t_m},\tilde G_n^+,\tilde G_n^-,\tilde R_n,\tilde N_n^*)$ has the same distribution (with respect
to $\tilde{\mathbb P}$) as
$(F_{n,t_1},\dots, F_{n,t_m},G_n^+,G_n^-,R_{\lfloor nt_m\rfloor},N_{\lfloor nt_m\rfloor}^*)$
(with respect to $\mathbb P$)
in $({\mathcal  C},D)^m\times (\mathbb R,\vert\cdot\vert)^4$;
\item $(\tilde F_{t_1},\dots,\tilde F_{t_m},\tilde G^+,\tilde G^-)$ has the same distribution as $(F_{t_1},\dots,F_{t_m},G^+,G^-)$
in $({\mathcal  C},D)^m\times (\mathbb R,\vert\cdot\vert)^4$;
\item the sequence $(\tilde F_{n,t_1},\dots,\tilde F_{n,t_m},\tilde G_n^+,\tilde G_n^-,n^{-1/\alpha-\varepsilon}\tilde R_n,n^{(1/\alpha)
-1-\varepsilon}\tilde N_n^*)_n$
converges almost surely to $(\tilde F_{t_1},\dots,\tilde F_{t_m},\tilde G^+,\tilde G^-,0,0)$ 
in $({\mathcal  C},D)^m\times(\mathbb R,|\cdot|)^4$.
\end{itemize}

Observe that, for every $x\in\mathbb Z$ and every $n\ge 1$,
$\mathfrak{N}_n(x):f\mapsto n(f((x+1)n^{-\frac 1\alpha})-f(xn^{-\frac 1\alpha}))$ is a continuous functional of 
$({\mathcal  C},D)$ and that 
$N_{\lfloor nt_i\rfloor}(x)=\mathfrak{N}_n(x)(F_{n,t_i})$ (for every
$i\in\{1,\dots,m\}$).
Therefore, for every integers $x$ and $n\ge 1$, for every
$i\in\{1,\dots,m\}$, we define 
$$\tilde N_{n,t_i}(x):=\mathfrak{N}_n(x)(\tilde F_{n,t_i}) .$$
Observe that, for every integer $N\ge 1$,
$$\left((\tilde N_{n,t_i}(x))_{x\in\{-N,\dots,N\};i\in\{1,\dots,m\}},\tilde N_n^*,\tilde R_n,\tilde G_n^\pm\right)$$
has the same distribution as
$$\left(( N_{\lfloor nt_i\rfloor}(x))_{x\in\{-N,\dots,N\};i\in\{1,\dots,m\}}, N_{\lfloor nt_m\rfloor}^*, R_{\lfloor nt_m\rfloor}, G_n^\pm\right).$$
In particular $\tilde N_{n,t_i}(x)$ takes integer values
and $0\le\tilde N_{n,t_i}(x)\le\tilde N_{n,t_m}(x)$. Moreover we have the following result.

\begin{lem}\label{LEMMA}
Let $n$ be a positive integer.
We have
\begin{equation}\label{eqNn}
\sup_{x\in\mathbb Z}\tilde N_{n,t_m}(x)\le\tilde N_n^*,
\end{equation}
\begin{equation}\label{eqRn}
\#\{x\in\mathbb Z:\tilde N_{n,t_m}(x)>0\}=\tilde R_n
\end{equation}
and
\begin{equation}\label{eqGn}
\tilde G_n^\pm= n^{-2\beta\delta}\sum_{x,y\in\mathbb Z}\left|\sum_{i=1}^m\theta_i\tilde N_{ n,t_i}(x)\tilde N_{ n,t_i}(y)\right|_\pm^\beta.
\end{equation}
\end{lem}

\begin{proof}
(\ref{eqNn}) comes from the fact that, for every integers $x$ and $n\ge 1$, 
$ \tilde N_n^*-\tilde N_{n,t_m}(x)$ has the same distribution as
$ N_{\lfloor nt_m\rfloor}^*-N_{\lfloor nt_m\rfloor}(x)$ which is non negative.

To prove (\ref{eqRn}), we observe that
$$\tilde R_n-\#\{x\in\mathbb Z\ :\ \tilde N_{n,t_m}(x)>0\}
  =\lim_{N\rightarrow+\infty}\left(
    \tilde R_n-\#\big\{x\in\{-N,\dots,N\}\ :\ 
  \tilde N_{n,t_m}(x)>0\big\}\right) .$$
But, for every $N\ge 1$,
$\tilde R_n-\#\big\{x\in\{-N,\dots,N\}\ :\ 
  \tilde N_{n,t_m}(x)>0\big\}$ has the same distribution as
$R_{\lfloor nt_m\rfloor}-\#\big\{x\in\{-N,\dots,N\}\ :\  N_{\lfloor nt_m\rfloor}(x)>0\big\}$
which converges to 0 as $N$ goes to infinity.
This gives (\ref{eqRn}) by uniqueness of the limit for the convergence
in probability.

Finally, we observe that
$\tilde G_n^\pm- n^{-2\beta\delta}\sum_{x,y\in\mathbb Z}\left|\sum_{i=1}^m\theta_i\tilde N_{n,t_i}(x)\tilde N_{ n,t_i}(y)\right|_\pm^\beta$
is the limit as $N$ goes to infinity of
$$\tilde G_n^\pm- n^{-2\beta\delta}\sum_{|x|,|y|\le N}\left|\sum_{i=1}^m\theta_i\tilde N_{ n,t_i}(x)\tilde N_{ n,t_i}(y)\right|_\pm^\beta$$
which has the same distribution as
$$ G_n^\pm- n^{-2\beta\delta}\sum_{|x|,|y|\le N}\left|\sum_{i=1}^m\theta_i N_{ \lfloor nt_i\rfloor}(x) N_{\lfloor nt_i\rfloor}(y)\right|_\pm^\beta.$$
But this last random variable converges to 0 as $N$ goes to infinity and we obtain (\ref{eqGn}).
\end{proof}

Let us write $(\Omega,\mathcal F,\mathbb P)$ for the original space
on which $\xi$ and $S$ are defined. We denote $\mathcal F_\xi$ for the
sub-$\sigma$-algebra of $\mathcal F$ generated by $\xi$ and $ \mathbb{P}_{\xi} $ for the restriction of $ \mathbb{P} $ to $ \mathcal F_\xi $ .
Now we define $({\bf {\bf \Omega}},{\bf \mathcal T},{\bf P})$ as the direct product of
$(\Omega,\mathcal F_\xi,\mathbb P_\xi)$ with
$(\tilde\Omega,\tilde{\mathcal F},\tilde{\mathbb P})$.
We observe that $ \mathbb P_\xi(\cdot)=\mathbf P(\cdot | \tilde{\mathcal F}) $.

\begin{lem}\label{memeloi}
For every integer $n\ge 1$, the random variable
$\tilde{\mathfrak{A}}_n:=\sum_{x,y\in\mathbb Z}\sum_{i=1}^m\theta_i \tilde N_{n,t_i}(x)
\tilde N_{n,t_i}(y)h(\xi_x,\xi_y)$ has the same distribution (with respect to
${\mathbf P}$) as
${\mathfrak{A}}_n:=\sum_{x,y\in\mathbb Z}\sum_{i=1}^m\theta_i  N_{\lfloor nt_i\rfloor}(x)N_{\lfloor nt_i\rfloor}(y)h(\xi_x,\xi_y)$ (with respect to $\mathbb P$).
\end{lem}

\begin{proof}
We proceed as in the proof of Lemma \ref{LEMMA}.
Observe that $\tilde{\mathfrak{A}}_n$ is the limit
as $N$ goes to infinity of $\tilde{\mathfrak{A}}_{n,N}:=\sum_{|x|,|y|\le N}\sum_{i=1}^m\theta_i \tilde N_{n,t_i}(x) \tilde N_{n,t_i}(y)h(\xi_x,\xi_y)$
which has the same distribution as 
$\mathfrak{A}_{n,N}:=\sum_{|x|,|y|\le N}\sum_{i=1}^m\theta_i  N_{\lfloor nt_i\rfloor}(x) N_{\lfloor nt_i\rfloor}(y)h(\xi_x,\xi_y)$.
But $\mathfrak{A}_n=\lim_{N\rightarrow+\infty}\mathfrak{A}_{n,N}$.
We conclude by unicity of the limit for the convergence in distribution.
\end{proof}

Let $\tilde\Omega_0\subset\tilde\Omega$ be the set of $ \tilde{\mathbb P} $-measure  one on which  $(\tilde F_{n,t_1},\dots,\tilde F_{n,t_m},\tilde G_n^+,\tilde G_n^-,n^{-1/\alpha-\varepsilon}\tilde R_n,
  n^{(1/\alpha)-1-\varepsilon}\tilde N_n^*)_n$
converges to $(\tilde F_{t_1},\dots,\tilde F_{t_m},\tilde G^+,\tilde G^-,0,0)$ in $ \mathcal C ^m\times \mathbb R^4 $.

\subsection{A conditional limit theorem for some associated point process}\label{STEP2}

To simplify notations, we set
\begin{equation}\label{nota}
\zeta_{n,x,y}:=\sum_{i=1}^m\theta_i\tilde N_{n,t_i}(x)\tilde N_{n,t_i}(y)\ \ \mbox{if}\ \ \alpha_0>1
\end{equation}
and
\begin{equation}\label{nota0}
\zeta_{n,x,y}:=\sum_{i,j=1}^m\theta_{i,j}d_{i,n}(x)d_{j,n}(y)\ \ \mbox{if}\ \  \alpha_0= 1.
\end{equation}
With these notations we have
$$\tilde G_n^\pm=a_n^{-\beta}\sum_{x,y}\left|\zeta_{n,x,y}\right|_\pm^\beta.$$
For every $\tilde\omega\in\tilde\Omega_0$, we consider the point process ${\mathcal N}_n$
on $\mathbb R^*$ defined by
$${\mathcal N}_n(\tilde\omega,\xi)(dz):=\sum_{x,y\in\mathbb Z:x\ne y} \delta_{a_n^{-1}\zeta_{n,x,y}(\tilde\omega)h(\xi_x,\xi_y)}(dz).$$
We already mentioned in (\ref{anequiv}) that $a_n\sim cn^2(\mathbb E[R_n])^{\frac 2\beta-2}$ for some
$c>0$ and observe that in any case 
\begin{equation}\label{an-1}
\forall\gamma_0>0,\ \ a_n^{-1}=o\left( n^{-2+\frac 2{\alpha_0}-
  \frac 2{\alpha_0\beta}+\gamma_0}\right).
\end{equation}
Moreover note that for the $\epsilon>0$ which was fixed in the previous subsection we have 
\begin{eqnarray*}
 n^{\frac{1}{\alpha_0}-1-\epsilon} \tilde{N}^\ast_n\stackrel{a.s.}{\longrightarrow} 0 
\end{eqnarray*}
and
\begin{eqnarray*}
 n^{-\frac{1}{\alpha_0}-\epsilon}\tilde{R}_n\stackrel{a.s.}{\longrightarrow} 0 .
\end{eqnarray*}

In the following we will prove that the sequence of point processes $ {\mathcal N}_n; n\in\mathbb{N} $ converges toward some 
Poisson point process for $ \tilde{\mathbb{P}} $ almost all $ \tilde{\omega}\in\tilde{\Omega} $.
We will essentially follow the notation from \cite{Resnick87} and denote by $ M_p(\mathbb{R}^\ast) $ the set of point measures on 
$ \mathbb{R}^\ast $. 
Further, $ {\mathcal M}_p(\mathbb{R}^\ast) $ is the  smallest $ \sigma $-algebra containing all sets $ A $  of the form 
$$ A=\{m\in M_p(\mathbb{R}^\ast); m(F)\in B\} $$ 
for some $ F\in{\mathcal B}(\mathbb{R}^\ast) $ and $ B\in{\mathcal B}([0,\infty]) $.
We introduce the following metric on $ \mathbb{R}^\ast $
\[
 d(x,y):=\left\{\begin{array}{cc}   
      |\log(x/y)| & \mbox{if $ {\rm sgn}(x)={\rm sgn}(y) $};\\
      |\log |x|| + |\log |y|| + 1 & \mbox{if  $ {\rm sgn}(x)\neq{\rm sgn}(y) $}.
\end{array}\right.
\]
With this metric $ \mathbb{R}^\ast $ becomes a complete separable metric space.
We will denote by $ C_K(\mathbb{R}^\ast) $ the space of continuous functions $ f:\mathbb{R}^\ast\rightarrow\mathbb{R} $ 
with compact support with respect to this metric. A sequence of Radon measures $ \mu_n $ is said to converge 
with respect to the vague topology toward some Radon measure $ \mu $ if for all $ f\in  C_K(\mathbb{R}^\ast) $ one has
$$   \lim_{n\rightarrow\infty}\int_{\mathbb{R}^\ast} f d\mu_n =\int_{\mathbb{R}^\ast} f d\mu  .$$
It is well known that the vague topology on the Radon measures can be generated by some metric which turns it into a complete metric space (see \cite{Resnick87} p.147) and that the set of point measures is closed in the vague topology (see \cite{Resnick87} p.145). We will say that a sequence of point processes $  {\mathcal N}_n;n\in\mathbb{N} $ converges in distribution toward a Point process $ {\mathcal N} $ if for all bounded vaguely continuous functions $ F: M_p(\mathbb{R}^\ast)\rightarrow\mathbb{R} $ we have
$$ \lim_{n\rightarrow\infty}\mathbb{E}[F({\mathcal N}_n)]=\mathbb{E}[F({\mathcal N})] .$$

\begin{prop}\label{PROP1}
For every $\tilde\omega\in\tilde\Omega_0$,
${\mathcal N}_n(\tilde\omega,\cdot)$ converges in distribution
(with respect to $\mathbb P_\xi$) to a Poisson process ${\mathcal N}_{\tilde\omega}$
on $\mathbb R\setminus\{0\}$ 
of intensity $\eta_{\tilde\omega}$ given by
$$\eta_{\tilde\omega}([d,d'))
=({d}^{-\beta}-{d'}^{-\beta})\frac{(c_0+c_1)\tilde G^+(\tilde\omega)+(c_0-c_1)\tilde G^-(\tilde\omega)}2,$$
and
$$\eta_{\tilde\omega}((-d',-d])
=({d}^{-\beta}-{d'}^{-\beta})\frac{(c_0+c_1)\tilde G^+(\tilde\omega)-(c_0-c_1)\tilde G^-(\tilde\omega)}2,$$
(with convention $\infty^{-\beta}=0$)
for every $0<d<d'\le +\infty$.
\end{prop}

\begin{proof}
Our proof is based on some method presented in \cite{DDMS}. 
Due to Kallenberg's theorem \cite{Resnick87}, it is enough to prove that, for any
finite union $R=\bigcup_{i=1}^ KQ_i$ of intervals, where $ Q_i:=[d_i,d'_i)\subset(0,+\infty)$
or $Q_i=(-d'_i,-d_i]\subset(-\infty,0)$. We have
\begin{equation}\label{Kallenberg1}
\lim_{n\rightarrow +\infty}\mathbf E[{\mathcal N}_n(R)|\tilde{\mathcal F}](\tilde\omega)=   \eta_{\tilde\omega}(R)
\end{equation}
and
\begin{equation}\label{Kallenberg2}
\lim_{n\rightarrow +\infty}\mathbf P({\mathcal N}_n(R)=0|\tilde{\mathcal F})(\tilde\omega)=e^{-\eta_{\tilde\omega}(R)}.
\end{equation}
We start with the proof of (\ref{Kallenberg1}). By linearity, it is enough to prove it for a single interval $ Q $. 
For any interval
$ Q=[d,d')\subset(0,+\infty)$, since $\xi$ is a sequence of iid random variables, we have
$$\mathbf E[{\mathcal N}_n(Q)|\tilde{\mathcal F}]=\sum_{x,y\in\mathbb Z^{d_0}:x\ne y}
\left(\mathbf P( \mathcal A_{n,x,y}|\tilde{\mathcal F})\mathbf 1_{\{\zeta_{n,x,y}>0\}}+\mathbf P(\mathcal B_{n,x,y}|\tilde{\mathcal F})\mathbf 1_{\{\zeta_{n,x,y}<0\}}\right),$$
with
$$\mathcal A_{n,x,y}:=\Big\{a_n d|\zeta_{n,x,y}|^{-1}\le h(\xi_1,\xi_2)<
  a_n d'|\zeta_{n,x,y}|^{-1}\Big\}$$
and
$$\mathcal B_{n,x,y}:=\Big\{a_n d|\zeta_{n,x,y}|^{-1}\le -h(\xi_1,\xi_2)<
  a_n d'|\zeta_{n,x,y}|^{-1}\Big\}.$$
Observe that, due to (\ref{an-1}) and to $\tilde N_n^*=o(n^{1-\frac 1{\alpha_0}+
\varepsilon})$, we have
\begin{equation}\label{MAJO}
\forall\gamma_0>0,\ \ a_n^{-1}\sup_{x,y}|\zeta_{n,x,y}|\le C a_n^{-1}(\tilde N_{ n}^*)^2\le
    n^{-\frac 2{\alpha_0\beta}+2\varepsilon+\gamma_0},
\end{equation}
for $n$ large enough (and for some constant $C>0$ depending on $\theta_i$ or on
$\theta_{i,j}$).
Now, combining this with Item (iii) of Assumption \ref{HYP}, we have 
\begin{eqnarray*}
\sum_{x,y:x\ne y}
\mathbf P( \mathcal A_{n,x,y}|\tilde{\mathcal F})\mathbf 1_{\{\zeta_{n,x,y}>0\}}
&=& c_0 ({d}^{-\beta}-{d'}^{-\beta}) 
a_n^{-\beta}\sum_{x,y\in\mathbb Z^{d_0}:x\ne y}\left|\zeta_{n,x,y}\right|^\beta 
\frac{\mbox{sgn}(\zeta_{n,x,y})+1}2\\
&\ &\ \ \ \ \ \ \ 
\times\left(1+O\left(\sup_{z>n^{\frac 2{\alpha_0\beta}-2\varepsilon-\gamma_0}}|L_0(z)-c_0|\right)\right)+o(1)\\
&=&  c_0 ({d}^{-\beta}-{d'}^{-\beta})\frac{\tilde G_n^++\tilde G_n^-}2+o(1),
\end{eqnarray*}
since $\varepsilon<1/(\alpha_0\beta)$ and since, for $n$ large enough,
$$\sum_{x\in\mathbb Z^{d_0}} 
  \left|\zeta_{n,x,y}\right|^\beta\le n^{\frac 1{\alpha_0}+\varepsilon}
   n^{2\beta-\frac{2\beta}{\alpha_0}+2\varepsilon\beta}=o(a_n^\beta)
,$$
since $\varepsilon<1/((1+2\beta)\alpha_0)$.
Analogously, we have
\begin{eqnarray*}
\sum_{x,y:x\ne y}
\mathbf P( \mathcal B_{n,x,y}|\tilde{\mathcal F})\mathbf 1_{\{\zeta_{n,x,y}<0\}}
&=& c_1 ({d}^{-\beta}-{d'}^{-\beta}) 
a_n^{-\beta}\sum_{x,y\in\mathbb Z^{d_0}:x\ne y}\left|\zeta_{n,x,y}\right|^\beta 
\frac{1-\mbox{sgn}(\zeta_{n,x,y})}2\\
&\ &\ \ \ \ \ \ \ 
\times\left(1+O\left(\sup_{z>n^{\frac 2{\alpha_0\beta}-2\varepsilon-\gamma_0}}|L_1(z)-c_1|\right)\right)+o(1)\\
&=& c_1  ({d}^{-\beta}-{d'}^{-\beta})\frac{\tilde G_n^+-\tilde G_n^-}2+o(1),
\end{eqnarray*}
We obtain (\ref{Kallenberg1}) for $Q=[d,d')\subset(0,+\infty)$ using (\ref{hyph1}), (\ref{hyph1b}) and the definition
of $\tilde G_n^\pm$ and of $\tilde G^\pm$. 
The proof of (\ref{Kallenberg1}) for $Q=(-d',-d]\subset(-\infty,0)$ follows the same scheme.

Now let us prove (\ref{Kallenberg2}). Let $K\ge 1$ and let $R$ be a union
of $K$ pairwise disjoint intervals $Q_1,...,Q_K$ with $Q_i:=(d_i,d'_i]\subset(0,+\infty)$ 
or $Q_i:=[-d'_i,-d_i)\subset(-\infty,0)$.
We write $P_n^{\tilde\omega}$ for the Poisson distribution of intensity $\eta_n^{\tilde\omega}(R)
:=\mathbf E[ {\mathcal N}_n(R)|\tilde{\mathcal F}](\tilde\omega)$.
On $\tilde\Omega_0$, due to (\ref{Kallenberg1}), we have
$$|e^{-\eta_{\tilde\omega}(R)} - P_n^{\tilde\omega}(0)|=o(1).$$
Hence, to prove (\ref{Kallenberg2}), we just have to prove
\begin{equation}
|\mathbf P({\mathcal N}_n(R)=0|\tilde{\mathcal F})-
      P_n(0)|=o(1).
\end{equation}
Following \cite{BarbourEagleson} and \cite{DDMS}, we 
introduce the following notations.
For every $x,y\in \mathbb Z^{d_0}$ such that $x\ne y$, 
we define the random variables
$$I_{x,y}=\sum_{i=1}^K
 \mathbf{1}_{\{ h(\xi_x,\xi_y)\in a_n 
  (\zeta_{n,x,y})^{-1}Q_i\}}.$$
Observe that 
\begin{equation}\label{NandI}
{\mathcal N}_n(R)=\sum_{x,y\in \mathbb Z^{d_0}:x\ne y} I_{x,y}
\ \ \mbox{and so}\ \ \eta_n(R)=\sum_{x,y\in \mathbb Z^{d_0}:x\ne y}\mathbf E[I_{x,y}|\tilde{\mathcal F}].
\end{equation}
We will use the following lemma, whose proof is postponed until the end of this paragraph:

\begin{lem}\label{LEM2}
We have
$$|\mathbf P({\mathcal N}_n(R)=0|\tilde{\mathcal F})-
      P_n(0)|\le \min(1,(\eta_n(R))^{-1})(A_1+A_2),$$
with 
$$A_1:= \sum_{(x,y)\in M}
\mathbf E[I_{x,y}|\tilde{\mathcal F}]\mathbf E\left[\left.I_{x,y}+\sum_{(x',y')\in M_{x,y}^{(1)}}I_{x',y'}\right|\tilde{\mathcal F}\right],$$
$$A_2:= \sum_{(x,y)\in M}\mathbf E\left[\left.I_{x,y}\left(I_{x,y}+
\sum_{(x',y')\in M_{x,y}^{(1)}} I_{x',y'}\right)\right|\tilde{\mathcal F}\right],$$
and with the notation 
$ M_{x,y}^{(k)}:=\{(x',y')\in M: \#\{x',y'\}\cap\{x,y\}=k\}$ and $ M:=\{(x,y)\in\mathbb{Z}^{2d_0}: x\neq y\} $.
\end{lem}

To conclude, we have to prove that $A_1$ and $A_2$ converge to $0$ as $n$
goes to infinity.

We set $d:=\min_id_i$.

For $A_1$, using (\ref{hyph1}), (\ref{hyph1b}) and the definition of $I_{x,y}$, we observe that,
for $\gamma_0>0$ small enough, we have
\begin{eqnarray*}
A_1&\le&4
  \sum_{x,y\in \mathbb Z^{d_0}}\sum_{x'\in\mathbb Z^{d_0}}
 \mathbf P\left( da_n\left|\zeta_{n,x,y}\right|^{-1}\le |h(\xi_x,\xi_y)|\Big| \tilde{\mathcal F} \right)\\
&\ &\ \ \ \ \ \ \ \ \ \ \ \ \ \ \ \ \ \times 
\mathbf P\left( da_n\left|\zeta_{n,x,x'}\right| 
  ^{-1}\le |h(\xi_x,\xi_{x'})|\Big| \tilde{\mathcal F} \right)\\
&\le&C d^{-2\beta}a_n^{-2\beta}(\Vert L_0\Vert_\infty+\Vert L_1\Vert_\infty)^2  
  \tilde R_n^3(\tilde N_n^*)^{4\beta}\\
&\le&O(n^{-\frac 1{\alpha_0} +(4\beta+3)\varepsilon+\gamma_0})=o(1),
\end{eqnarray*}
using $\varepsilon(4\beta+3)<1/\alpha_0$, (\ref{an-1}) together with 
the definitions of $\tilde R_n$ and $\tilde N_n^*$
(with $C$ some constant depending on $\theta_j$ and $\theta_{i,j}$).

Now let us study $A_2$. We have, for $\gamma_0>0$ small enough,
\begin{eqnarray*}
A_2&\le& 4\sum_{x,y,x'\in\mathbb Z^{d_0}}
\mathbf P\left(da_n\left|\zeta_{n,x,y}\right|^{-1}\le |h(\xi_x,\xi_y)|,
da_n\left|\zeta_{n,x,x'}\right|^{-1}\le |h(\xi_x,\xi_{x'})|\Big| \tilde{\mathcal F} \right)\\
&\le&4C_0 \tilde R_n^3a_n^{-2\gamma}(\tilde N_n^*)^{4\gamma}\\
&\le&O\left(n^{\frac 3{\alpha_0}+(3+4\gamma)\varepsilon-\frac{4\gamma}{\alpha_0\beta}+\gamma_0}\right)=o(1),
\end{eqnarray*}
due to $(3+4\gamma)\varepsilon\alpha_0<\frac{4\gamma}{\beta}-3$ (recall that
this is possible since $\gamma>3\beta/4$) and where $C_0$ 
is a constant depending on on $d$, $\theta_j$ and $\theta_{i,j}$.
\end{proof}

\begin{proof}[Proof of Lemma \ref{LEM2}]
The proof of this lemma follows the line of arguments that can be found in \cite{DDMS}. 
Let $f$ be defined on $\mathbb N$ by $f(0)=0$ and
$$f(m):=e^{\eta_n(R)}\frac{(m-1)!}{(\eta_n(R))^{m}}P_n(\{0\})P_n([m,+\infty)).$$
We will use the two following inequalities (see \cite{BarbourEagleson} p.400 and p.401)
\begin{equation}\label{BE1}
\Big|\mathbf P({\mathcal N}_n(R)=0|\tilde{\mathcal F})-
      P_n(0)\Big|\le\Big|\mathbf E\Big[\eta_n(R)f({\mathcal N}_n(R)+1)-{\mathcal N}_n(R)f({\mathcal N}_n(R))\Big|\tilde{\mathcal F}\Big]\Big|
\end{equation}
and
\begin{equation}\label{BE2}
\sup_m|f(m+1)-f(m)|\le \min(1,(\eta_n(R))^{-1}).
\end{equation}
Now we observe that, for every $(x,y)\in(\mathbb Z^{d_0})^2$
such that $x\ne y$, we have
\begin{equation} \label{NnR} 
{\mathcal N}_n(R)=\sum_{x',y'\in\mathbb Z^{d_0}:x'\ne y'}
I_{x',y'}=I_{x,y}+{\mathcal N}_{n,x,y}^{(0)}+
{\mathcal N}_{n,x,y}^{(1)},
\end{equation}
with ${\mathcal N}_{n,x,y}^{(i)}:=\sum_{(x',y')\in M_{x,y}^{(i)}}I_{x',y'}$.
Starting from (\ref{BE1}) and using (\ref{NandI}), we have
$$
\Big|\mathbf P({\mathcal N}_n(R)=0|\tilde{\mathcal F}) - P_n(0)\Big|\le A'_1+A'_2,$$
with
$$A'_1:=\left|\sum_{x,y\in\mathbb Z^{d_0}:x\ne y}\mathbf E[I_{x,y}|\tilde{\mathcal F}]
  \mathbf E\Big[f({\mathcal N}_n(R)+1)-f({\mathcal N}_{n,x,y}^{(0)}+1)\Big|\tilde{\mathcal F}\Big]\right| $$
and
$$A'_2:=\left|\sum_{x,y\in\mathbb Z^{d_0}:x\ne y}\mathbf E\Big[I_{x,y}f({\mathcal N}_n(R))\Big|\tilde{\mathcal F}\Big]
              -\mathbf E[I_{x,y}|\tilde{\mathcal F}]
  \mathbf E\Big[f({\mathcal N}_{n,x,y}^{(0)}+1)\Big|\tilde{\mathcal F}\Big]\right| .$$
Now, using (\ref{BE2}) and (\ref{NnR}), we obtain
\begin{eqnarray}
\left|f({\mathcal N}_n(R)+1)-f({\mathcal N}_{n,x,y}^{(0)}+1)\right|
&\le& \sup_{m\ge 0}|f(m+1)-f(m)|\times\left({\mathcal N}_n(R)
    -{\mathcal N}_{n,x,y}^{(0)}\right)\nonumber\\
&\le& \min(1,(\eta_n(R))^{-1})(I_{x,y}+{\mathcal N}_{n,x,y}^{(1)})\label{BE3}
\end{eqnarray}
and so $A'_1\le \min(1,(\eta_n(R))^{-1})A_1$.
Observe that, conditioned with respect to $\tilde{\mathcal F}$,
$I_{x,y}$ and ${\mathcal N}_{n,x,y}^{(0)}$ are independent. Therefore
$$A'_2=\left|\sum_{x,y\in\mathbb Z^{d_0}:x\ne y}\mathbf E\Big[I_{x,y}\{f({\mathcal N}_n(R))-
   f({\mathcal N}_{n,x,y}^{(0)}+1)\}\Big|\tilde{\mathcal F}\Big]\right|.$$
Now, using (\ref{BE2}) once again, we obtain
\begin{eqnarray*}
\left|f({\mathcal N}_n(R))-f({\mathcal N}_{n,x,y}^{(0)}+1)\right|
&\le&
   \min(1,(\eta_n(R))^{-1})({\mathcal N}_n(R)-{\mathcal N}_{n,x,y}^{(0)})\\
&\le&
   \min(1,(\eta_n(R))^{-1})(I_{x,y}+{\mathcal N}_{n,x,y}^{(1)})\\
\end{eqnarray*}
and so $A'_2\le \min(1,(\eta_n(R))^{-1})A_2$, which completes the proof of the lemma.
\end{proof}

\subsection{Proof of the convergence of the finite dimensional distributions}\label{STEP3}

In this paragraph we will finish the proof of the convergence of the finite dimensional distributions. Similarly to the proof given in \cite{DDMS}, we will use the convergence of the associated point process and the continuous mapping theorem. The approach is based on the following observation: 
$$
a_n^{-1}\sum_{x,y} \zeta_{n,x,y}h(\xi_x,\xi_y)=\int_{\mathbb R^*}w\, 
d {\mathcal N}_n(w).$$
However the functional is not continuous and we will have to do some truncation. 
This will be the purpose of the three following propositions.

\begin{prop}\label{PPP1}Let $\delta>0$.
For $ \tilde{\mathbb{P}} $ almost every $\tilde\omega\in\tilde\Omega_0$, the sequence of random variables
$$
Z_n^{\tilde \omega}:=a_n^{-1}\sum_{x,y} \zeta_{n,x,y}(\tilde\omega)h(\xi_x,\xi_y)\mathbf 1_{\{a_n^{-1}|\zeta_{n,x,y}(\tilde\omega)h(\xi_x,\xi_y)|> \delta\}}=\int_{\mathbb R^*}w\mathbf 1_{(\delta,+\infty)}(|w|)\, 
d {\mathcal N}_n^{\tilde\omega}(w)$$
converges in distribution to
$\int_{\mathbb R^*}w\mathbf 1_{(\delta,+\infty)}(|w|)\, 
d {\mathcal N}^{\tilde\omega}(w).$
\end{prop}

\begin{prop}\label{PPP2}
For every $\gamma_0>0$, we have
$$\lim_{\delta\rightarrow 0}\limsup_{n\rightarrow\infty}
   \mathbf P\left(|T_n(\delta)|>\gamma_0|\tilde{\mathcal F}\right)=0 \ \ \ \tilde{\mathbb P}-a.s.,$$
with
$$T_n(\delta):=a_n^{-1}\sum_{x,y} \zeta_{n,x,y}h(\xi_x,\xi_y)\mathbf 1_{\{a_n^{-1}|\zeta_{n,x,y}h(\xi_x,\xi_y)|\le \delta\}}\ \ \mbox{if}\ \ \beta\le 1$$
and
$$T_n(\delta):=a_n^{-1}\sum_{x,y} \zeta_{n,x,y}h(\xi_x,\xi_y)\mathbf 1_{\{a_n^{-1}|\zeta_{n,x,y}h(\xi_x,\xi_y)|\le \delta\}}+(c_0-c_1)
\frac{\beta\delta^{1-\beta}}{\beta-1}\tilde G_n^-
\ \ \mbox{if}\ \ \beta>1.$$
\end{prop}

\begin{prop}[see \cite{SmorodnitskyTaqqu}]\label{PPP3}
Let $\mathcal P$ be a Poisson process on $\mathbb R^*$ with intensity admitting the
density $z\mapsto\beta|z|^{-\beta-1}(a\mathbf{1}_{\{z>0\}}+b\mathbf{1}
_{\{z<0\}})$. 

If $\beta<1$, then $\int_{\mathbb R^*\setminus[-\delta,\delta]}w\, d\mathcal P(w)$
converges in distribution, as $\delta$ goes to 0, to
a stable random variable with characteristic function $\Phi_{a+b,a-b,\beta}$
with the notation of (\ref{FonCar}).

If $\beta=1$, then $\int_{\mathbb R^*\setminus[-\delta,\delta]}w\, d\mathcal P(w)-(a-b)\int_\delta^{+\infty}\frac{\sin x}{x^2}\, dx$
converges in distribution, as $\delta$ goes to $0$, to
a stable random variable with characteristic function
$\Phi_{a+b,a-b,1}$, with the notation of (\ref{FonCar1}).

If $\beta>1$, then 
$\int_{\mathbb R^*\setminus[-\delta,\delta]}w\, d\mathcal P(w)-(a-b)\frac
{\beta\delta^{1-\beta}}{\beta-1}$
converges in distribution, as $\delta$ goes to 0, to
a stable random variable with characteristic function $\Phi_{a+b,a-b,\beta}$
with the notation of (\ref{FonCar}).
\end{prop}

The following corollary is a consequence of Propositions \ref{PROP1},
 \ref{PPP1}, \ref{PPP2} and \ref{PPP3}.

\begin{cor}\label{PPP4}
We have
$$\lim_{n\rightarrow+\infty}
\mathbf E[e^{ia_n^{-1}\sum_{x,y}\zeta_{n,x,y}(\tilde\omega)h(\xi_x,\xi_y)}|\tilde{\mathcal F}] =\Phi_{(c_0+c_1)\tilde G^+(\tilde\omega),(c_0-c_1)\tilde G^-(\tilde\omega),\beta}(1),$$
for $ \tilde{\mathbb P} $-almost every $\tilde\omega$ in $\tilde \Omega $ 
and
$$\lim_{n\rightarrow+\infty}
\mathbf E\left[e^{ia_n^{-1}\sum_{x,y}\zeta_{n,x,y}h(\xi_x,\xi_y)}\right] =
\mathbf E\left[\Phi_{(c_0+c_1)\tilde G^+,(c_0-c_1)\tilde G^-,\beta}(1)\right].$$
\end{cor}

\begin{proof}[Proof of Corollary \ref{PPP4}]
Observe first that due to the Lebesgue dominated convergence theorem
it is enough to prove the first convergence.
Let $\tilde\Omega_1$ be the subset of $\tilde\Omega_0$ on which the convergences of Propositions \ref{PPP1}
and \ref{PPP2} hold and let $\tilde\omega\in\tilde\Omega_1$.
To simplify notations, let us write 
$$V_{n}
:=a_n^{-1}\sum_{x,y}\zeta_{n,x,y}h(\xi_x,\xi_y)\ \ \mbox{and}\ \ W_{n}(\delta):=a_n^{-1}\sum_{x,y}\zeta_{n,x,y}h(\xi_x,\xi_y)
\mathbf 1_{\{a_n^{-1}|\zeta_{n,x,y}
h(\xi_x,\xi_y)|>\delta\}}.$$
We set $\kappa:=0$ if $\beta\le 1$ and $\kappa:=(c_0-c_1)\frac\beta{\beta-1}$ if $\beta>1$ (recall that we assume $c_0=c_1$ if
$\beta=1$). We also write $W_{\tilde\omega}(\delta):=\int_{\mathbb R\setminus[-\delta,\delta]}w\, d{\mathcal N}_{\tilde\omega}(w)$ (where $\mathcal N_{\tilde\omega}$ is the Poisson process of Proposition \ref{PROP1}, which is defined on some probability
space $ (\Omega_{\tilde\omega},\mathcal T_{\tilde\omega}, {\rm P}_{\tilde\omega})$ endowed with the expectation $ {\rm E}_{\tilde\omega}$).
Let $\epsilon>0$. 
Due to Propositions \ref{PPP2}, \ref{PROP1} and \ref{PPP3}, 
we consider $\delta>0$ and $n_0$ such that,
for every $n\ge n_0$, we have
\begin{equation}\label{AAA1}
\mathbf P\left(|T_n(\delta)|>\frac{\epsilon}6\Big|\tilde{\mathcal F}\right)(\tilde\omega)<\frac\epsilon 6
\end{equation}
and such that
\begin{equation}\label{AAA2}
\Big|{\rm E}_{\tilde\omega}\big[e^{i(W_{\tilde\omega}
   (\delta)-\kappa\delta^{1-\beta}\tilde G^-(\tilde\omega))}
\big ]-\Phi_{(c_0+c_1)\tilde G^+(\tilde\omega),(c_0-c_1)\tilde G^-(\tilde\omega),\beta}(1)\Big|<\frac\epsilon 6 .
\end{equation}
Due to Proposition \ref{PPP1}, we consider $n_1\ge n_0$ such that, for
every $n\ge n_1$, we have
\begin{equation}\label{AAA3}
\left|\mathbf E[e^{iW_n(\delta)}|\tilde{\mathcal F}]
(\tilde\omega)-
{\rm E}_{\tilde\omega}[e^{iW_{\tilde\omega}(\delta)}]\right|<\frac\epsilon 6.
\end{equation}
Now, let $n_2\ge n_1$ such that, for every $n\ge n_2$, we have
\begin{equation}\label{AAA4}
\Big|e^{i\kappa\delta^{1-\beta}\tilde G^-(\tilde\omega)}-e^{i\kappa\delta^{1-\beta}\tilde G_n^-(\tilde\omega)}\Big|<\frac\epsilon 6.
\end{equation}
For $n\ge n_2$ , we have
\begin{eqnarray*}
 && \displaystyle\left|\mathbf E[e^{iV_n}|\tilde{\mathcal F}](\tilde\omega) -\Phi_{(c_0+c_1)\tilde G^+(\tilde\omega),(c_0-c_1)\tilde G^-(\tilde\omega),\beta}(1)\right|\l \\
&\le&\frac\epsilon 6+\left|\mathbf E[e^{iV_n}|\tilde{\mathcal F}](\tilde\omega)-{\rm E}_{\tilde\omega}
[e^{i(W_{\tilde\omega}(\delta)-\kappa\delta^{1-\beta}\tilde G^-(\tilde\omega))}]\right|\ \ \ \mbox{due to (\ref{AAA2})}\\
&\le& \frac\epsilon 6+\left|\mathbf E[e^{i(V_n+\kappa\delta^{1-\beta}\tilde G^-)}|\tilde{\mathcal F}](\tilde\omega)-{\rm E}_{\tilde\omega}
[e^{i(W_{\tilde\omega}(\delta)}]\right|\\
&\le& \frac{2\epsilon} 6+\left|\mathbf E[e^{i(V_n+\kappa\delta^{1-\beta}\tilde G_n^-)}|\tilde{\mathcal F}](\tilde\omega)
-{\rm E}_{\tilde\omega}
[e^{i(W_{\tilde\omega}(\delta)}]\right|\ \ \ \mbox{due to (\ref{AAA4})}\\
&\le& \frac{2\epsilon} 6+\left|\mathbf E[e^{i(W_n(\delta)+T_n(\delta))}|\tilde{\mathcal F}](\tilde\omega)-{\rm E}_{\tilde\omega}
[e^{i(W_{\tilde\omega}(\delta)}]\right|\\
&\le& \frac{3\epsilon} 6+\left|\mathbf E[e^{i(W_n(\delta)+T_n(\delta))}
-e^{iW_n(\delta)}|\tilde{\mathcal F}](\tilde\omega)\right|\ \ \ \mbox{due to (\ref{AAA3})}\\
&\le& \frac{4\epsilon} 6+2\mathbf P\left(\left.|T_n(\delta)|>\frac{\epsilon}6\right|\tilde{\mathcal F}\right)(\tilde\omega)\le\epsilon\ \ \ \mbox{due to (\ref{AAA1})}.
\end{eqnarray*}
\end{proof}

\begin{proof}[Proof of the convergence of finite distributions in Theorems
\ref{THM0}, \ref{THM00} and \ref{THM1}]
Admitting Propositions \ref{PPP1}, \ref{PPP2} and \ref{PPP3} 
for the moment, let us
end the proof of the convergence of the finite distributions.
Due to Corollary \ref{PPP4}, we have
\begin{eqnarray*}
\lim_{n\rightarrow+\infty}\mathbf E[e^{ia_n^{-1}\sum_{x,y}\zeta_{n,x,y}
     h(\xi_x,\xi_y)}]
  &=&\mathbf E\left[\Phi_{(c_0+c_1)\tilde G^+,(c_0-c_1)\tilde G^-,\beta}(1)
 \right]\\
  &=&\mathbb E\left[\exp\left(-\int_0^{+\infty}\frac{\sin t}{t^\beta}\, dt
\left[(c_0+c_1) G^+-i(c_0-c_1) G^-\tan\frac{\pi\beta}2 \right]\right)\right].
\end{eqnarray*}
When $\alpha_0= 1$, with the use of (\ref{F1}) and (\ref{blaa}) , we obtain

\begin{eqnarray*}
 &&  \lim_{n\rightarrow +\infty}\mathbb E\left[e^{ia_n^{-1}\sum_{j=1}^m\theta_{j}
   (U_{\lfloor nt_j\rfloor}-U_{\lfloor nt_{j-1}\rfloor})}\right]\\
&=&\exp\left(-K_\beta^2\sum_{i=1}^m(t_i^2-t_{i-1}^2)
    |\theta_i|^\beta \int_0^{+\infty}
    \frac{\sin t}{t^\beta}\, dt\left[(c_0+c_1)-i(c_0-c_1)\sgn(\theta_i)\tan\frac{\pi\beta}2\right]\right)\\
&=&\prod_{j=1}^m\Phi_{(c_0+c_1)K_\beta^2(t_i^2-t_{i-1}^2),(c_0-c_1)K_\beta^2(t_i^2-t_{i-1}^2),\beta}(\theta_j)
\end{eqnarray*}
This gives the convergence of the finite distributions in Theorems
\ref{THM0} and \ref{THM00}.

When $\alpha_0>1$, due to Lemma \ref{memeloi}, we obtain
\begin{equation}\label{FonctionCar}
\lim_{n\rightarrow +\infty}\mathbb E\left[e^{i\sum_{j=1}^m\theta_{j}
   a_n^{-1}U_{\lfloor nt_j\rfloor}}\right]
      =\mathbb E\left[\Phi_{(c_0+c_1)G^+,(c_0-c_1)G^-,\beta}(1)\right],
\end{equation}
with $G^\pm=\int_{\mathbb R^2}\left|\sum_{i=1}^m\theta_i\mathcal L_{t_i}(x)\mathcal L_{t_i}(y)\right|^\beta_\pm\, dxdy$.
Let us recall that the right hand side of 
(\ref{FonctionCar}) corresponds to the characteristic
function of $\sum_{i=1}^m\theta_i\int_{\mathbb R^2}\mathcal L_{t_i}(x)\mathcal L_{t_i}(y)\,
dZ_{x,y}$ evaluated at one (see for example \cite{KhoshnevisanNualart} and
Appendix \ref{calculsto}).
\end{proof}

\begin{proof}[Proof of Proposition \ref{PPP1}]
To simplify notations we also write $ {\rm P}_{\tilde\omega} $ for $ {\bf P}(\cdot|\tilde{\mathcal F})(\tilde\omega) $ and 
$ {\rm E}_{\tilde\omega} $ for $ {\bf E}[\cdot|\tilde{\mathcal F}](\tilde\omega) $.

We proceed in four steps: \\
1) We first use the continuous mapping theorem (see \cite{Resnick87} p.151) to prove that for 
$ \tilde{\mathbb P} $-almost all $ \tilde{\omega} $ one has
\begin{eqnarray}  \label{lawconvergence}
\int_{(-M,-\delta)\cup(\delta,M)}zd{\mathcal N}_n^{\tilde{\omega}}(dz)\stackrel{\mathcal L}{\longrightarrow}
       \int_{(-M,-\delta)\cup(\delta,M)}zd{\mathcal N}^{\tilde{\omega}}(dz) .\end{eqnarray}
The Poisson process $ \tilde{\mathcal N}_{\tilde{\omega}} $ has $ \tilde{\mathbb{P}} $-almost surely only a finite number of points 
in the interval $ (-M,-\delta)\cup (\delta,M) $. Moreover,  one has $ \tilde{\mathbb{P}} $-almost surely that each of those
points only carries the mass one, since the Poisson process $  \tilde{\mathcal N}_{\tilde{\omega}} $ is simple. 
Now, let $ \mu $ be a point measure with only a finite number of points with mass one in 
$ (-M,-\delta)\cup(\delta,M) $ and let $ (\mu_n)_{n\in\mathbb{N}} $ be some sequence of point measures which converges toward $ \mu $ 
with respect to the vague topology on $ \mathbb{R}^\ast $. Let $ \{x_1,...,x_p\} $ be the support of $ \mu $ 
intersected with $ (-M,-\delta)\cup (\delta,M) $. According to \cite{Neveu1976} (see Lemma I.14) there exists some large $ N\in \mathbb{N} $ such that for all $ n\geq N $ the support of $ \mu_n $ intersected with $  (-M,-\delta)\cup (\delta,M) $ in exactly $ p $ point $ x_1^{(n)},...,x_p^{(n)} $ such that 
$$ \lim_{n\rightarrow\infty}x_i^{(n)}=x_i \ \ \ \mbox{for all $ i=1,...,p$.} $$ 
It then follows that 
\begin{eqnarray*}
  \lim_{n\rightarrow\infty}\int_{(-M,-\delta)\cup(\delta,M)}z\mu_n(dz)
   &=& \lim_{n\rightarrow\infty}\sum_{i=1}^p x_i^{(n)} 
   = \sum_{i=1}^p x_i 
  = \int_{(-M,-\delta)\cup(\delta,M)}z\mu(dz) .
\end{eqnarray*}
2) We now prove that for $ \tilde{\mathbb P} $-almost all $ \tilde{\omega} $ one has 
\begin{eqnarray} \label{stochasticconvergence}
 \int_{(-\infty,-M)\cup(M,\infty)}zd{\mathcal N}^{\tilde{\omega}}(dz)\stackrel{ {\rm P}_{\tilde\omega}}{\longrightarrow}0  \ \ \ \mbox{as $ M\rightarrow\infty $}.
\end{eqnarray}
This follows from the following equality which holds for $ \tilde{\mathbb P} $-almost all $ \tilde{\omega} $ 
\begin{eqnarray*}
   {\rm E}_{\tilde\omega}\left[\exp\left(it\int_M^\infty z {\mathcal N}^{\tilde{\omega}}(dz)\right)\right]
     = \exp\left((c_0+c_1)\tilde{G}^+\int_M^\infty\beta\frac{\cos(tx)-1}{x^{\beta+1}}dx+i(c_0-c_1)\tilde{G}^-\int_M^\infty\beta\frac{\sin(tx)}{x^{\beta+1}}dx\right)
\end{eqnarray*}
and from the fact that one has
$$  \left|(c_0+c_1)\tilde{G}^+\int_M^\infty\beta\frac{\cos(tx)-1}{x^{\beta+1}}dx+i(c_0-c_1)\tilde{G}^-\int_M^\infty\beta\frac{\sin(tx)}{x^{\beta+1}}dx\right| \leq 2M^{-\beta}\left((c_0+c_1)(|\tilde G^+|+|\tilde G^-|)\right) .$$
This yields 
$$    {\rm E}_{\tilde\omega}\left[\exp\left(it\int_M^\infty z {\mathcal N}^{\tilde{\omega}}(dz)\right)\right] \longrightarrow1 \ \ \ \mbox{for $ \tilde{\mathbb{P}} $ almost all $ \tilde{\omega} $ as $ M\rightarrow\infty $.} $$
The convergence in probability follows from the convergence in law of $ \int_M^\infty z {\mathcal N}^{\tilde{\omega}}(dz) $ toward zero. The other part $ \int_{-\infty}^{-M} z {\mathcal N}^{\tilde{\omega}}(dz) $ is treated in the same way.\\[1mm]
3) We now prove that  for $ \tilde{\mathbb P} $-almost all $ \tilde{\omega} $ we have
\begin{eqnarray} \label{uniformconvergence}
 \sup_{n\in\mathbb{N}} {\rm P}_{\tilde\omega}\left(\int_{(-\infty,-M)\cup(M,\infty)} z {\mathcal N}^{\tilde{\omega}}_n(dz)\neq0\right) \longrightarrow0 
  \ \ \   \mbox{as $  M\rightarrow\infty $.}
 \end{eqnarray}
For this first remember that
$$  \int_{(-\infty,-M)\cup(M,\infty)} z {\mathcal N}^{\tilde{\omega}}_n(dz)
    =\sum_{x,y\in\mathbb{Z}}a_n^{-1}\zeta_{n,x,y}h(\xi_x,\xi_y)
       {\bf 1}_{\{|a_n^{-1}\zeta_{n,x,y}h(\xi_x,\xi_y)|>M\}}  .$$
Thus this implies
\begin{eqnarray*}
  {\rm P}_{\tilde\omega}\left(\int_{\{|z|>M\}} z {\mathcal N}^{\tilde{\omega}}_n(dz)\neq0\right)
  &\leq &  {\rm P}_{\tilde\omega}\Big(\exists x,y\in\mathbb{Z}: \ |a_n^{-1}\zeta_{n,x,y}h(\xi_x,\xi_y)|>M\Big) \\
   &\leq&  \sum_{x,y\in\mathbb{Z}} {\rm P}_{\tilde\omega}\Big(|h(\xi_x,\xi_y)|>Ma_n|\zeta_{n,x,y}|^{-1}\Big) \\
   &\leq& \sum_{x,y\in\mathbb{Z}}C\Big(Ma_n|\zeta_{n,x,y}|^{-1}\Big)^{-\beta}\\
  &\leq& CM^{-\beta}a_n^{-\beta}\sum_{x,y\in\mathbb{Z}}|\zeta_{n,x,y}|^{\beta} = CM^{-\beta}G_n^+\longrightarrow0
 \ \ \ \mbox{as $ M\rightarrow\infty $},
\end{eqnarray*}
since $ \mathbb{P} $-almost surely we have  $ G_n^+\rightarrow G^+ $ as $ n\rightarrow\infty $.\\[1mm]
4) We now use the previous findings to conclude. We consider an $ \tilde\omega $ which satisfies all the requirements from points  (1) to (3) of this proof.  
For some given $ t\in\mathbb{R} $ and $ \epsilon>0  $ we use (\ref{uniformconvergence}) to find some $ M>0 $ such that
$$  \sup_{n\in\mathbb{N}} {\rm P}_{\tilde\omega}\left(\int_{(-\infty,-M)\cup(M,\infty)} z {\mathcal N}^{\tilde{\omega}}_n(dz)\neq0\right)\leq \epsilon/8  $$
By (\ref{stochasticconvergence}) we can assume without loss of generality that the $ M $ also satisfies
$$   {\rm P}_{\tilde\omega}\left(t\left|\int_{(-\infty,-M)\cup(M,\infty)}zd\mathcal N^{\tilde{\omega}}(dz)\right|\geq\epsilon/4\right)\leq\epsilon/8 . $$
Moreover, according to (\ref{lawconvergence}) we can find some $ n_0\in\mathbb{N} $ such that for all $ n\geq n_0 $ we have
$$  \left| {\rm E}_{\tilde\omega}\left[\exp\left(it\int_{(-M,-\delta)\cup(\delta,M)}z{\mathcal N}_n^{\tilde{\omega}}(dz)\right)\right]
  -  {\rm E}_{\tilde\omega}\left[\exp\left(it\int_{(-M,-\delta)\cup(\delta,M)}z{\mathcal N}^{\tilde{\omega}}(dz)\right)\right]\right|\leq\epsilon/4. $$
It now follows that
\begin{eqnarray*}
 && \left| {\rm E}_{\tilde\omega}\left[\exp\left(it\int_{(-\infty,-\delta)\cup(\delta,\infty)}z{\mathcal N}_n^{\tilde{\omega}}(dz)\right)\right]
  -  {\rm E}_{\tilde\omega}\left[\exp\left(it\int_{(-\infty,-\delta)\cup(\delta,\infty)}z{\mathcal N}^{\tilde{\omega}}(dz)\right)\right]\right| \\
 &=& \Bigg| {\rm E}_{\tilde\omega}\Bigg[\exp\left(it\int_{(-M,-\delta)\cup(\delta,M)}z{\mathcal N}_n^{\tilde{\omega}}(dz)\right)
   \left(1+\exp\left(it\int_{(-\infty,-M)\cup(M,\infty)}z{\mathcal N}_n^{\tilde{\omega}}(dz)\right)-1\right)\Bigg] \\
  &&  -  {\rm E}_{\tilde\omega}\Bigg[\exp\left(it\int_{(-M,-\delta)\cup(\delta,M)}z{\mathcal N}^{\tilde{\omega}}(dz)\right)
   \left(1+\exp\left(it\int_{(-\infty,-M)\cup(M,\infty)}z{\mathcal N}^{\tilde{\omega}}(dz)\right)-1\right)\Bigg] \Bigg| \\
&\leq& \left|  {\rm E}_{\tilde\omega}\Bigg[\exp\left(it\int_{(-M,-\delta)\cup(\delta,M)}z{\mathcal N}_n^{\tilde{\omega}}(dz)\right)\Bigg]
   -  {\rm E}_{\tilde\omega}\Bigg[\exp\left(it\int_{(-M,-\delta)\cup(\delta,M)}z{\mathcal N}^{\tilde{\omega}}(dz)\right)\Bigg]\right| \\
  &&+ 2 {\rm P}_{\tilde\omega}\left(\int_{(-\infty,-M)\cup(M,\infty)} z {\mathcal N}^{\tilde{\omega}}_n(dz)\neq0\right)
   +  2 {\rm P}_{\tilde\omega}\left(t\left|\int_{(-\infty,-M)\cup(M,\infty)}zd\mathcal N^{\tilde{\omega}}(dz)\right|\geq\epsilon/4\right)+\frac{\epsilon}{4}.
\end{eqnarray*}
Since the right side is equal to $ \epsilon $ this finishes the proof of the proposition.
\end{proof}
\begin{proof}[Proof of Proposition \ref{PPP2}]
\begin{itemize}
\item When $\beta<1$, we just prove that $\lim_{\delta\rightarrow 0}\limsup_{n\rightarrow\infty}\mathbf E[|T_n(\delta)||\tilde{\mathcal F}]=0$.
Due to Item (iii) of Assumption \ref{HYP}, we have
\begin{eqnarray*}
\mathbf E[|T_n(\delta)||\tilde{\mathcal F}]&\le&
\sum_{x,y} \mathbf E\Big[a_n^{-1}|\zeta_{n,x,y}h(\xi_x,\xi_y)|\mathbf 1_{\{a_n^{-1}|\zeta_{n,x,y}h(\xi_x,\xi_y)|\le \delta\}}\Big|\tilde{\mathcal F}\Big]\\
&\le&\sum_{x,y}
      \int_0^{\delta}\mathbf P\Big(\delta\ge a_n^{-1}|h(\xi_x,\xi_y)\zeta_{n,x,y}|>z\Big|\tilde{\mathcal F}\Big)\, dz\\
&\le&\sum_{x,y}
      \int_0^{\delta}\mathbf P\Big(|h(\xi_x,\xi_y)\zeta_{n,x,y}|>a_nz\Big|\tilde{\mathcal F}\Big)\, dz\\
&\le& (\Vert L_0\Vert_\infty+\Vert L_1\Vert_\infty)
\sum_{x,y} \int_0^{\delta}a_n^{-\beta}z^{-\beta}(\zeta_{n,x,y})^\beta\, dz\\
&\le& (\Vert L_0\Vert_\infty+\Vert L_1\Vert_\infty)\sum_{x,y}\frac{a_n^{-\beta}\delta^{1-\beta}}{1-\beta}
      (\zeta_{n,x,y})^\beta\\
&\le& (\Vert L_0\Vert_\infty+\Vert L_1\Vert_\infty)\frac{\delta^{1-\beta}}{1-\beta}\tilde G_n^+.
\end{eqnarray*}
So 
$\lim_{\delta\rightarrow 0}\limsup_{n\rightarrow\infty}\mathbf E[|T_n(\delta)||\tilde{\mathcal F}]\le \lim_{\delta\rightarrow 0}(
\Vert L_0\Vert_\infty+\Vert L_1\Vert_\infty)\frac{\delta^{1-\beta}}{1-\beta}\tilde G^+=0,$
since $\beta<1$.

\item Assume here that $\beta\in(1,2)$.
Observe that, due to Item (v) of Assumption \ref{HYP}, we have

\begin{eqnarray*}
&& \mathbb E\big[h(\xi_1,\xi_2)\mathbf{1}_{\left\{|h(\xi_1,\xi_2)|\le M \right\}}\big]\\
&=&-\mathbb E\left[h(\xi_1,\xi_2)\mathbf{1}_{\{|h(\xi_1,\xi_2)|>M\}}\right]\\
&=&\int_{0}^{+\infty}\mathbb P\Big(h(\xi_1,\xi_2)<-\max(z,M)\Big)\,dz
-\int_{0}^{+\infty}\mathbb P\Big(h(\xi_1,\xi_2)>\max(z,M)\Big)\,dz\\
&=&M\Big(\mathbb P(h(\xi_1,\xi_2)<-M)-\mathbb P(h(\xi_1,\xi_2)>M)\Big)+\int_{M}^{+\infty}\mathbb P(h(\xi_1,\xi_2)<-z)\,dz\\
&\ &\ \ \ \ \ \ \ \ \ \ 
-\int_{M}^{+\infty}\mathbb P(h(\xi_1,\xi_2)>z)\,dz.
\end{eqnarray*}
But, due to Item (iii) of Assumption \ref{HYP}, as $x$ goes to infinity, we have
$$\mathbb P(h(\xi_1,\xi_2)>x)=c_0 x^{-\beta}+o(x^{-\beta}),\ \ 
\mathbb P(h(\xi_1,\xi_2)<-x)= c_1 x^{-\beta}+o(x^{-\beta}), $$
$$\int_x^{+\infty}\mathbb P(h(\xi_1,\xi_2)>z)\, dz= c_0\frac{x^{1-\beta}}{\beta-1}+o(x^{1-\beta}), $$
$$ \int_x^{+\infty}\mathbb P(h(\xi_1,\xi_2)<-z)\, dz= c_1\frac{x^{1-\beta}}{\beta-1}+o(x^{1-\beta})$$
and
$$\forall x>0,\ \ 
 \int_x^{+\infty}\Big(\mathbb P\big(h(\xi_1,\xi_2)>z\big)+\mathbb P\big(h(\xi_1,\xi_2)<-z\big)\Big)
\, dz\le (\Vert L_0\Vert_\infty+\Vert L_1\Vert_\infty)\frac{x^{1-\beta}}{\beta-1}.$$
Therefore, we obtain
\begin{equation}\label{SYM}
\mathbb E\left[h(\xi_1,\xi_2)\mathbf{1}_{\{|h(\xi_1,\xi_2)|\le M \}}\right]=M^{1-\beta}\left(\frac{\beta }{\beta-1}(c_1-c_0)+\epsilon_M\right),
\end{equation}
where $\lim_{M\rightarrow+\infty}\epsilon_M=0$ and $\sup_{M>0}\epsilon_M<\infty$.

\item When $\beta=1$, due to Item (vii) of Assumption \ref{HYP}, we have $c_0=c_1$ and  (\ref{SYM}) holds also true.
\item Assume now that $\beta\in[1,2)$.
We will prove that $\lim_{\delta\rightarrow 0}\limsup_{n\rightarrow\infty}\mathbf E[(T_n(\delta))^2|\tilde{\mathcal F}]=0$.
We have
$$ \mathbf E[(T_n(\delta))^2|\tilde{\mathcal F}]
=\sum_{x,y,x',y'\in\mathbb Z^{d_0}}{\mathbf E}[T_{n,x,y}T_{n,x',y'}|\tilde{\mathcal F}],$$
with 
$$T_{n,x,y}:=a_n^{-1}h(\xi_x,\xi_y)\zeta_{n,x,y}
{\mathbf 1}_{\{|h(\xi_x,\xi_y)\zeta_{n,x,y}|\leq a_n\delta\}}
+a_n^{-\beta}(c_0-c_1)\frac{\beta\delta^{1-\beta}}{\beta-1} |\zeta_{n,x,y}|_-^\beta$$
(recall that $c_0=c_1$ when $\beta=1$).

\begin{itemize}

\item \underline{ Contribution of $(x,y,x',y')$ such that $\{x,y\}\cap\{x',y'\}=\emptyset$.}

We set $E_1$ for the set of such $(x,y,x',y')$. Let $(x,y,x',y')\in E_1$. 
Since $h(\xi_{x},\xi_y)$ and $h(\xi_{x'},\xi_{y'})$ are independent conditionally to $\tilde{\mathcal F}$, we have
$${\mathbf E}[T_{n,x,y}T_{n,x',y'}|\tilde{\mathcal F}]={\mathbf E}[T_{n,x,y}|\tilde{\mathcal F}]
{\mathbf E}[T_{n,x',y'}|\tilde{\mathcal F}].$$
Now, due to (\ref{SYM}), we have
$$
\left|\sum_{x,y\in\mathbb Z^{d_0}}\mathbf E[T_{n,x,y}|\tilde{\mathcal F}]\right|\leq      \delta^{1-\beta}\sum_{x,y\in\mathbb Z^{d_0}}a_n^{-\beta}
|\zeta_{n,x,y}|_+^{\beta}\epsilon_{a_n\delta|\zeta_{n,x,y}|^{-1}}.
$$
Now, due to (\ref{MAJO}), for every $\gamma_0>0$, if $n$ is large enough,
we have
$$a_n^{-1}\sup_{x,y\in\mathbb Z^{d_0}}|\zeta_{n,x,y}|\le
    n^{-\frac 2{\alpha_0\beta}+2\varepsilon+\gamma_0}.$$
Combining this with $\lim_{n\rightarrow +\infty}\tilde G_n^+=\tilde G^+$
and with $\lim_{M\rightarrow+\infty}\epsilon_M=0$, we obtain
\begin{equation}\label{SUM1}
\limsup_{n\rightarrow +\infty}
 \sum_{x,y\in\mathbb Z^{d_0}}\mathbf E[T_{n,x,y}|\tilde{\mathcal F}]=0,
\end{equation}
 since $ \beta\epsilon < 1/\alpha_0 $.
This implies
$$\forall\delta>0,\ \ \limsup_{n\rightarrow +\infty}\sum_{(x,y,x',y')\in E_1} {\mathbf E}[T_{n,x,y}T_{n,x',y'}|\tilde{\mathcal F}]=0.$$

\item \underline{Contribution of $(x,y,x',y')$ such that $\{x,y\}=\{x',y'\}$.}

Let us write $E_2$ for the set of such $(x,y,x',y')$. Observe that

$$ \sum_{(x,y,x',y')\in E_2}{\mathbf E}[T_{n,x,y}T_{n,x',y'}|\tilde{\mathcal F}]\le 2\sum_{x,y\in\mathbb Z^{d_0}}{\mathbf E}[T_{n,x,y}^2|\tilde{\mathcal F}].$$

First, using Item (iii) of Assumption \ref{HYP}, we notice that

\begin{eqnarray*}
 && a_n^{-2}\sum_{x,y\in\mathbb Z^{d_0}}\mathbf E\Big[(h(\xi_1,\xi_2)\zeta_{n,x,y})^2\mathbf 
1_{\{|h(\xi_1,\xi_2)\zeta_{n,x,y}|\le a_n\delta\}}\Big|\tilde{\mathcal F}\Big] \\
&=& \sum_{x,y\in\mathbb Z^{d_0}}\int_0^{\delta^2}\mathbb P\Big(\sqrt{z}<a_n^{-1}|h(\xi_1,\xi_2)\zeta_{n,x,y}|<\delta\Big|\tilde{\mathcal F}\Big)\, dz\\
&\le& \sum_{x,y\in\mathbb Z^{d_0}}\int_0^{\delta^2}\mathbb P\Big(\sqrt{z}<a_n^{-1}|h(\xi_1,\xi_2)\zeta_{n,x,y}|\Big|\tilde{\mathcal F}\Big)\, dz\\
&\le&(\Vert L_0\Vert_\infty+\Vert L_1\Vert_\infty)
\sum_{x,y\in\mathbb Z^{d_0}}\int_0^{\delta^2}a_n^{-\beta}z^{-\frac{\beta}2}|\zeta_{n,x,y}|^\beta\, dz\\
&\le&(\Vert L_0\Vert_\infty+\Vert L_1\Vert_\infty)a_n^{-\beta}
\sum_{x,y\in\mathbb Z^{d_0}}|\zeta_{n,x,y}|^\beta \frac{\delta^{2\left(1-\frac\beta 2\right)}}{1-\frac\beta 2}\\
&\le&(\Vert L_0\Vert_\infty+\Vert L_1\Vert_\infty)\tilde G_n^+ \frac{\delta^{2-\beta}}{1-\frac\beta 2}.
\end{eqnarray*}
Therefore
\begin{equation}
\lim_{\delta\rightarrow 0}\limsup_{n\rightarrow+\infty}a_n^{-2}\sum_{x,y\in\mathbb Z^{d_0}}\mathbf E\Big[(h(\xi_1,\xi_2)\zeta_{n,x,y})^2\mathbf 
1_{\{|h(\xi_1,\xi_2)\zeta_{n,x,y}|\le a_n\delta\}}\Big|\tilde{\mathcal F}\Big]
=0.
\end{equation}

Second, using (\ref{an-1}) and the definition of $\tilde N_n^*$ and $\tilde R_n$, for every $\gamma_0>0$, for $n$ large enough,  we have
\begin{eqnarray*}
 a_n^{-2\beta}\left|\sum_{x,y\in\mathbb Z^{d_0}}\left((c_0-c_1)^2\frac{\beta^2\delta^{2-2\beta}}{(\beta-1)^2}|\zeta_{n,x,y}|_-^{2\beta}\right) \right|
&\le& (c_0-c_1)^2\frac{\beta^2\delta^{2-2\beta}}{(\beta-1)^2} a_n^{-2\beta}\tilde R_n^2  (\tilde N_n^\ast)^{4\beta}\\
&\le& n^{-\frac{2}{\alpha_0}+2\epsilon+4\beta\epsilon+\gamma_0}\delta^{2-2\beta}.
\end{eqnarray*}
So, since  $ \epsilon>0 $ satisfies $ (3+4\beta)\epsilon<\frac{1}{\alpha_0} $ we have that
\begin{equation}
\lim_{\delta\rightarrow 0}\limsup_{n\rightarrow+\infty}
a_n^{-2\beta}\sum_{x,y\in\mathbb Z^{d_0}}\left((c_0-c_1)^2\frac{\beta^2\delta^{2-2\beta}}
{(\beta-1)^2}|\zeta_{n,x,y}|_-^{2\beta}\right)=0.
\end{equation}
Finally this shows
$$\lim_{\delta\rightarrow 0}\limsup_{n\rightarrow+\infty}
\sum_{(x,y,x',y')\in E_2}{\mathbf E}[T_{n,x,y}T_{n,x',y'}|\tilde{\mathcal F}]=0.$$

\item \underline{Contribution of $(x,y,x',y')$ such that $\#(\{x,y\}\cap\{x',y'\})=1$.}

Let us write $E_3$ for the set of such $(x,y,x',y')$.
Observe that we have
$$\displaystyle \sum_{(x,y,x',y')\in E_3}{\mathbf E}[T_{n,x,y}T_{n,x',y'}|\tilde{\mathcal F}]=4 \sum_{x,y,z:x\ne y,x\ne z,y\ne z}\mathbb E\left[T_{n,x,y}T_{n,x,z}|\tilde{\mathcal F}\right]$$

\begin{itemize}
\item Assume that $1\le\beta<4/3$. 
We set $U_{n,x,y}:=a_n^{-1}h(\xi_x,\xi_y)\zeta_{n,x,y}
{\mathbf 1}_{\{|h(\xi_x,\xi_y)\zeta_{n,x,y}|\le a_n\delta\}}$.
Observe that
\begin{equation}\label{EEEE1}
T_{n,x,y}=U_{n,x,y}+a_n^{-\beta}(c_0-c_1)\frac{\beta\delta^{1-\beta}}{\beta-1}
|\zeta_{n,x,y}|_-^\beta
\end{equation}
(recall that we assume $c_0=c_1$ if $\beta=1$)
and that, due to (\ref{SYM}),
\begin{equation}\label{EEEE2}
\mathbb E[U_{n,x,y}|\tilde{\mathcal F}]=a_n^{-\beta}\delta^{1-\beta}|\zeta_{n,x,y}|_-^\beta
\left[(c_1-c_0)\frac{\beta}{\beta-1}+
\epsilon_{a_n\delta|\zeta_{n,x,y}|^{-1}}\right].
\end{equation}
Now, (\ref{MAJO}) ensures that 
\begin{equation}\label{EEEE3}
\lim_{n\rightarrow+\infty}\sup_{x,y}\epsilon_{a_n\delta|\zeta_{n,x,y}|^{-1}}=0.
\end{equation}
Moreover, we observe that, due to (\ref{an-1}) and to the definition of 
$\tilde N_n^*$ and of $\tilde R_n$, we have, for every $\gamma_0>0$
and every $n$ large enough,
\begin{eqnarray*}
\sum_{x,y,z\in\mathbb Z^{d_0}}a_n^{-2\beta}|\zeta_{n,x,y}|^\beta|\zeta_{n,x,z}|^\beta
&\le&\tilde R_n^3a_n^{-2\beta}\left(\tilde N_{n}^*\right)^{4\beta}\\
&\le& n^{-\frac{1}{\alpha_0}+3\epsilon+4\beta\varepsilon+\gamma_0}.
\end{eqnarray*}
Now, since $ (3+4\beta)\epsilon<\frac{1}{\alpha_0}  $ we conclude that
\begin{equation}\label{EEEE4}
\limsup_{n\rightarrow+\infty}\sum_{x,y,z\in\mathbb Z^{d_0}}a_n^{-2\beta}|\zeta_{n,x,z}|^\beta|\zeta_{n,x,y}|^\beta=0.
\end{equation}
Observe moreover that, due to Item (iv) of Assumption \ref{HYP}, we have

\begin{eqnarray*}
 && \mathbb E\left[|U_{n,x,y}U_{n,x,z}||\tilde{\mathcal F}\right]\\
&\le&\int_{(0,\delta)^2}\mathbb P(
a_n^{-1}|h(\xi_1,\xi_2)\zeta_{n,x,y}|>u,a_n^{-1}|h(\xi_1,\xi_3)\zeta_{n,x,z}|>v|\tilde{\mathcal F})\, dudv\\
&\le& C_0\left[a_n^{-1}|\zeta_{n,x,y}|+\int_{a_n^{-1}|\zeta_{n,x,z}|}^\delta
u^{-\gamma}a_n^{-\gamma}|\zeta_{n,x,y}|^\gamma\, du\right] \\
&& \times \left[a_n^{-1}|\zeta_{n,x,z}|+\int_{a_n^{-1}|\zeta_{n,x,z}|}^\delta
v^{-\gamma}a_n^{-\gamma}|\zeta_{n,x,z}|^\gamma\, dv\right]\\
&\le&C_0\left[a_n^{-1}|\zeta_{n,x,y}|^1+
\frac{\delta^{1-\gamma}-a_n^{\gamma-1}|\zeta_{n,x,z}|^{1-\gamma}}{1-\gamma}a_n^{-\gamma}|\zeta_{n,x,y}|^\gamma\right]\\
&\ &\ \ \times\left[a_n^{-1}|\zeta_{n,x,z}|+\frac{\delta^{1-\gamma}-
a_n^{\gamma-1}|\zeta_{n,x,z}|^{1-\gamma}}{1-\gamma}a_n^{-\gamma}|\zeta_{n,x,z}|^\gamma\right]\\
&\le&C_\delta a_n^{-2\gamma'}|\zeta_{n,x,y}\zeta_{n,x,z}|^{\gamma'}\ \ \mbox{where}\ \gamma'=\min(1,\gamma)
\end{eqnarray*}
for $n$ large enough and some $ C_\delta >0 $. Indeed, due to (\ref{MAJO}) we have 
$ a_n^{-1}\sup_{x,y}|\zeta_{n,x,y}|\le 1$ for large $ n $.  Again using (\ref{MAJO})
and to the definition of $\tilde R_n$, for every $\gamma_0>0$, we have

\begin{eqnarray*}
\displaystyle \sum_{x,y,z\in\mathbb Z^{d_0}}\mathbb E\left[|U_{n,x,y}U_{n,x,z}||\tilde{\mathcal F}\right]&\le& C_\delta\tilde R_n^3a_n^{-2\gamma'}\sup_{x,y}|\zeta_{n,x,y}|^{2\gamma'}\\
&\le&  n^{\frac 3{\alpha_0}-\frac{4\gamma'}{\alpha_0\beta}+7\varepsilon+\gamma_0},
\end{eqnarray*}
for $n$ large enough. Recall that we have chosen $\varepsilon$
such that $\frac 3{\alpha_0}-\frac{4\gamma'}{\alpha_0\beta}+7\varepsilon<0$.
Hence, we obtain
\begin{equation}\label{EEEE5}
\forall\delta>0,\ \ \ \limsup_{n\rightarrow+\infty}\sum_{x,y,z}\mathbb E[|U_{n,x,y}U_{n,x,z}|]=0.
\end{equation}
Now putting (\ref{EEEE1}), (\ref{EEEE2}), (\ref{EEEE3}), (\ref{EEEE4})
and (\ref{EEEE5}) all together, we conclude that
$$\forall\delta>0,\ \ \limsup_{n\rightarrow+\infty}
\displaystyle \sum_{(x,y,x',y')\in E_3}{\mathbf E}[T_{n,x,y}T_{n,x',y'}|\tilde{\mathcal F}]=0.$$
\item Assume now that $\beta\ge\frac 43$. Observe that,
with the notation of Item (vi) of Assumption \ref{HYP}, 
we have
$$T_{n,x,y}=a_n^{-1}\zeta_{n,x,y}{\mathbf h}_{(a_n\delta|\zeta_{n,x,y}|^{-1})}(\xi_x,\xi_y).$$
Due to this Item (vi), to the definition of $\tilde R_n$ and
to (\ref{MAJO}), for every $\gamma_0>0$, we have almost surely

\begin{eqnarray*}
\sum_{x,y,z\in\mathbb Z^{d_0}}|\mathbb E[T_{n,x,y}T_{n,x,z}|\tilde{\mathcal F}]|
&\le& C'_0a_n^{-2}\sum_{x,y,z\in \mathbb Z^{d_0}}|\zeta_{n,x,y}\zeta_{n,x,z}| \left(
   a_n^2\delta^2|\zeta_{n,x,y}\zeta_{n,x',y'}|^{-1}\right)^{-\theta'}\\
&\le& \delta^{-2\theta'}\tilde R_n^3 \left(a_n^{-1}(\tilde{N}_n^*)^2
\right)^{2(\theta'+1)}\\
&\le& n^{\frac 1{\alpha_0}\left(3-\frac {4(\theta'+1)}\beta\right)+
  (4\theta'+7)\varepsilon+\gamma_0},
\end{eqnarray*}
for $n$ large enough. Since $\frac 1{\alpha_0}\left(3-\frac {4(\theta'+1)}\beta\right)+
  (4\theta'+7)\varepsilon<0$, we obtain
$$\forall\delta>0,\ \ 
\limsup_{n\rightarrow +\infty}\sum_{(x,y,x',y')\in E_3}|\mathbb E[T_{n,x,y}T_{n,x,z}|\tilde{\mathcal F}]|=0.$$
\end{itemize}
\end{itemize}
So, finally, for $\beta\in[1,2)$, there exists $\tilde C>0$ such that, for
every nonnegative $n$ and every $\delta>0$, we have
$\limsup_{n\rightarrow+\infty}\mathbf E[(T_n(\delta))^2]\le \tilde C \delta^{2-\beta}$.
\end{itemize}
\end{proof}

\begin{proof}[Proof of Proposition \ref{PPP3}]
The following proof can be assembled from \cite{Feller}. 
We will use the constants $  I_0:=-\int_0^\infty\frac{\sin y}{y^\beta}\, dy $ and $ J_0:=-\tan\frac{\pi\beta}2I_0 $. 
Due to the exponential formula, we have
\begin{eqnarray*}
\mathbb E\left[e^{it\int_{\{|x|\ge\delta\}}x\, d\mathcal P(x)  }\right]&=&
\exp\left(\int_{\{|x|\ge\delta\}}(e^{itx}-1)(a\mathbf{1}_{\{x>0\}}+
 b\mathbf{1}_{\{x<0\}})\beta|x|^{-\beta-1}\, dx\right)\\
&=&\exp\left((a+b)\int_\delta^{+\infty}\frac{\cos(tx)-1}{x^{\beta+1}}
   \beta\, dx+i(a-b)\int_\delta^{+\infty}\frac{\sin(tx)}{x^{\beta+1}}\beta\, dx\right)
\end{eqnarray*}
Assume first that $\beta<1$. Due to \cite[p. 568]{Feller}, we have
$$ \lim_{\delta\rightarrow 0}\int_\delta^{+\infty}\frac{e^{itx}-1}{x^{\beta+1}}
\beta\, dx=-|t|^\beta\Gamma(1-\beta)e^{-\frac{i\pi\beta}2}=
|t|^\beta(I_0+iJ_0). $$
So
$\lim_{\delta\rightarrow 0}
\mathbb E\left[e^{it\int_{\{|x|\ge\delta\}}x\, d\mathcal P(x)  }\right]=
\Phi_{a+b,a-b,\beta}(t)$.\\
Assume now that $\beta=1$. Then 
$$\lim_{\delta\rightarrow 0}\int_\delta^{+\infty}\frac{\cos(tx)-1}{x^2} \, dx=\int_0^{+\infty}\frac{\cos(tx)-1}{x^2}\, dx
   =|t|\int_0^{+\infty}\frac{\cos(y)-1}{y^2}\, dy=-\frac{\pi}{2}|t|$$
and, since $\sin(tx)=\sgn(t)\sin(|t|x)$, we have
$$\int_\delta^{+\infty}\frac{\sin(tx)}{x^2}\, dx= t\int_{\delta |t|}
^{+\infty}\frac{\sin y}{y^2}\, dy$$
and so
$$\int_\delta^{+\infty}\frac{\sin(tx)}{x^2}\, dx- t\int_\delta
^{+\infty}\frac{\sin x}{x^2}\, dx=t\int_{\delta|t|}^{\delta}
\frac{\sin y}{y^2}\, dy\sim_{\delta\rightarrow 0}t\int_{\delta |t|}^{\delta}\frac{dy}y=-t\log|t|.$$
Hence we have in that case that
$$ \lim_{\delta\rightarrow0}\mathbb{E}\left[\exp\left(it\left(\int_{|x|>\delta}xd\mathcal P(x)-(a-b)\int_\delta^\infty\frac{\sin x}{x^2}dx\right)\right)\right] 
      =\Phi_{a+b,a-b,1}(t) .$$
Assume finally $\beta>1$. Due to \cite[p.568-569]{Feller}, we have
$$\lim_{\delta\rightarrow 0}\int_\delta^{\infty}\frac{e^{itx}-1-itx}{x^{\beta+1}}\, \beta\, dx = \int_0^{+\infty}\frac{e^{itx}-1-itx}{x^{\beta+1}} \beta\, dx
         =|t|^\beta\frac{\Gamma(3-\beta)e^{-\frac{i\pi\beta}2}}
  {(2-\beta)(\beta-1)}=|t|^\beta(I_0+iJ_0) .$$
So
$$\lim_{\delta\rightarrow 0}\mathbb E\left[
    e^{it\int_{\{|x|\ge\delta\}}x\, d{\mathcal P}(x)-it(a-b)\beta\frac{\delta^{1-\beta}}{\beta-1} }\right]=\Phi_{a+b,a-b,\beta}(t).$$
\end{proof}


\section{Tightness}\label{tightness}

Here we treat case $ \alpha_0>1 $ (i.e. the case where $(S_n)_n$
is recurrent and $\alpha>d_0=1$).
The tightness proof follows essentially the one given in Kesten and Spitzer \cite{kest}.
We need the following lemma from \cite{kest}.

\begin{lem}[Lemma 1 of \cite{kest}]\label{lem1KS}
For all $ \epsilon>0 $ there exists some $ A>0 $ such that for all $ t\geq 1 $ one has
$$ \mathbb{P}\left(\exists x\in\mathbb{Z}: |x|>At^{1/\alpha} \ \mbox{and} \ N_t(x)>0 \right) \leq \epsilon .$$
\end{lem}

\begin{lem} \label{QuadratAbschaetzungen}
We have
\begin{eqnarray}
  \mathbb{E}\left[\sum_{x\in\mathbb{Z}}N_n^2(x)\right] =  O(n^{2-\frac{1}{\alpha}}) \ \ \ \mbox{and} \ \ \ 
  \mathbb{E}\left[\Big(\sum_{x\in\mathbb{Z}}N_n^2(x)\Big)^2\right] = O(n^{4-\frac{2}{\alpha}}) .
\end{eqnarray}
\end{lem}

\begin{proof}
The first one is  formula (2.13) from \cite{kest} and the second one can be found in \cite[Lemma 2.1]{guil}.
\end{proof}

\begin{prop}
The sequence of stochastic processes 
$$ U_t^n:=n^{-2\delta}\sum_{x,y\in\mathbb{Z}}N_{\lfloor nt \rfloor}(x)N_{\lfloor nt \rfloor}(y)h(\xi_x,\xi_y); \ t\geq0 $$
is tight in $ C(0,T)  $ with sup-norm.
\end{prop}

\begin{proof}
It is sufficient to show that
$$   \limsup_{n\rightarrow\infty}\lim_{\kappa\downarrow0}\sup_{0\leq t_1,t_2\leq T: |t_1-t_2|\leq\kappa}\mathbb{P}\left(\left|U_{t_1}^n-U_{t_2}^n\right|>\eta\right)=0 .$$
Fix some $ \epsilon>0 $. Due to Lemma \ref{lem1KS}, we fix $ A>0 $ large enough such that
\begin{eqnarray} \label{Ungleichung1}
\mathbb{P} \Big(\exists x\in\mathbb{Z} \ \mbox{with $ |x|>An^{1/\alpha} $ and $  \ N_{\lfloor nT \rfloor}(x)>0$}  \  \Big) \leq\frac{\epsilon}{4}.
\end{eqnarray}
Choose some $ \rho>0 $ such that for all $ n\in\mathbb{N} $ one has
\begin{eqnarray} \label{Ungleichung2}
9A^2n^{2/\alpha}\mathbb P\left(|h(\xi_1,\xi_2)|>\rho n^{\frac{2}{\alpha\beta}}\right)
    <\frac{\epsilon}{4} .
\end{eqnarray}
This is possible since we have, by Item (iii) of Assumption \ref{HYP}, that
\begin{eqnarray} \label{Asymptotik}
  \lim_{u\rightarrow\infty}u^\beta\mathbb{P}(h(\xi_1,\xi_2)\geq u)=c_0 \ \ \ \mbox{and}
       \ \ \  \lim_{u\rightarrow\infty}u^\beta\mathbb{P}(h(\xi_1,\xi_2)\leq-u)=c_1.
\end{eqnarray}
Define 
\[
  \bar{h}(x,y):=h(x,y)\mathbf{1}_{\{| h(x,y)|\leq \rho n^{\frac{2}{\alpha\beta}}\}}.
\]
The inequality (\ref{Ungleichung2}) now becomes
\begin{equation} \label{Ungleichung3}
 9A^2n^{2/\alpha}\mathbb{P}\left(\bar{h}(\xi_1,\xi_2)\neq h(\xi_1,\xi_2)\right)  \leq \frac{\epsilon}{4} .
\end{equation}

\begin{lem}\label{majoesp}
There exists a constant $ C=C(\rho,\beta)>0 $ such that for all $ n\geq 1 $ one has
\begin{eqnarray}
       \Big| \mathbb{E}\left[\bar{h}(\xi_1,\xi_2)\right]\Big|\le Cn^{(1-\beta)\frac{2}{\alpha\beta}}.
\end{eqnarray}
\end{lem}

\begin{proof}
For $ \beta<1 $, we have
\begin{eqnarray*}
   \left|\mathbb{E}\left[\bar{h}(\xi_1,\xi_2)\right]\right| &\leq& 
   \int_0^{\rho n^{\frac{2}{\alpha\beta}}}\mathbb{P}\left(\left|h(\xi_1,\xi_2)\right|>x\right)dx 
\leq C\int_1^{\rho n^{\frac{2}{\alpha\beta}}}x^{-\beta}dx +1 \\
    &=& Cx^{1-\beta}\Bigg|_1^{\rho n^{2/\alpha\beta}} +1\sim C n^{\frac{2}{\alpha\beta}(1-\beta)}
\end{eqnarray*}
where $ C>0 $ is some suitable constant.
For $\beta\in(1,2)$, this comes from \eqref{SYM}.
For $\beta=1$, as noticed previously, this comes from Item (vii) of Assumption \ref{HYP}.
\end{proof}

Now we define 
$$ E_n:=n^{-2\delta}\mathbb{E}\left[\sum_{x,y\in\mathbb{Z}}N_n(x)N_n(y)\bar{h}(\xi_x,\xi_y)\right] .$$
Since the scenery and the random walk are independent, we compute
\begin{eqnarray*}
  E_n&=& n^{-2\delta}\mathbb{E}\left[\sum_{x,y\in\mathbb{Z}}N_n(x)N_n(y)\mathbb{E}\left[\bar{h}(\xi_x,\xi_y)\right]\right]
 =  n^{-2\delta}n^2\mathbb{E}\left[\bar{h}(\xi_1,\xi_2)\right]\\
    &\leq& Cn^{-2+\frac{2}{\alpha}-\frac{2}{\alpha\beta}}n^2n^{(1-\beta)\frac{2}{\alpha\beta}} = C,
\end{eqnarray*}
due to Lemma \ref{majoesp}.
Thus the sequence $ E_n $ stays bounded as $ n\rightarrow\infty $.
Further, let
\begin{eqnarray*}
  \bar{U}_t^n &:=& n^{-2\delta}\sum_{x,y\in\mathbb{Z}}N_{\lfloor nt \rfloor}(x)N_{\lfloor nt \rfloor}(y)\left(\bar{h}(\xi_x,\xi_y)-\mathbb{E}\left[\bar{h}(\xi_x,\xi_y)\right]\right)  .
\end{eqnarray*}
It then follows
\begin{eqnarray*}
  U_t^n-\bar{U}_t^n-t^2E_n &=& n^{-2\delta}\sum_{x,y\in\mathbb{Z}}N_{\lfloor nt \rfloor}(x)N_{\lfloor nt \rfloor}(y)\left(h(\xi_x,\xi_y)-\bar{h}(\xi_x,\xi_y)\right) \\
   && + \  n^{-2\delta} \Big(\lfloor nt \rfloor^2\mathbb{E}\left[\bar{h}(\xi_1,\xi_2)\right]-n^2t^2\mathbb{E}\left[\bar{h}(\xi_1,\xi_2)\right]\Big).
\end{eqnarray*}
Since we have that $ \mathbb{E}[\bar{h}(\xi_1,\xi_2)]=O(n^{(1-\beta)\frac{2}{\alpha\beta}}) $ and $ \lfloor nt \rfloor^2-n^2t^2=O(n)  $ 
the second term is of the order 
$$   n^{-2\delta}O(n^{(1-\beta)\frac{2}{\alpha\beta}})(\lfloor nt \rfloor^2-n^2t^2) 
     = n^{-2}O(n)=O(n^{-1}) .$$
This implies with inequalities (\ref{Ungleichung1}) and (\ref{Ungleichung3}) that 
\begin{eqnarray*}
 && \limsup_{n\rightarrow\infty}\mathbb{P}\left(\sup_{0\leq t\leq T}\left|U_t^n-\bar{U}_t^n-t^2E_n\right|>\frac{\eta}{2}\right) \\
   & \leq &  \limsup_{n\rightarrow\infty}\mathbb{P}\left(n^{-2\delta}\sum_{x,y\in\mathbb{Z}}N_{\lfloor nT \rfloor}(x)N_{\lfloor nT \rfloor}(y)\left(h(\xi_x,\xi_y)-\bar{h}(\xi_x,\xi_y)\right) >\frac{\eta}{4}\right)\\
 &\leq& \limsup_{n\rightarrow\infty}\mathbb{P}\left(\sum_{x,y\in\mathbb{Z}}N_{\lfloor nT \rfloor}(x)N_{\lfloor nT \rfloor}(y)\left(h(\xi_x,\xi_y)-\bar{h}(\xi_x,\xi_y)\right)\neq 0 \right)  \\
   &\leq & 
  \limsup_{n\rightarrow\infty} \mathbb{P}\left(\exists x,y\in\mathbb{Z}: |x|,|y|\leq An^{1/\alpha},\bar{h}(\xi_x,\xi_y)\neq h(\xi_x,\xi_y)\right) \\
 &&+  \ \limsup_{n\rightarrow\infty} \mathbb{P}\left(\exists x\in\mathbb{Z}: \ |x|>A n^{1/\alpha},N_{\lfloor nT \rfloor}(x)>0 \right) \\
  &\leq &  \limsup_{n\rightarrow\infty} \ (3An^{1/\alpha})^2\mathbb{P}\left(\bar{h}(\xi_1,\xi_2)\neq h(\xi_1,\xi_2)\right) 
          + \ \frac{\epsilon}{4}\\
&\leq& \frac{\epsilon}{2}.
\end{eqnarray*}
It is now sufficient to prove that 
$$   \limsup_{n\rightarrow\infty}\lim_{\kappa\downarrow0}\sup_{0\leq t_1,t_2\leq T: |t_1-t_2|\leq\kappa }\mathbb{P}\left(\left|\bar{U}_{t_1}^n-\bar{U}_{t_2}^n\right|>\frac{\eta}{2}\right)=0 .$$
For this we prove for all $ T\geq t> s\geq 0 $ that
\begin{eqnarray}\label{erwartungs-quadrat-abschaetzung}
   \mathbb{E}\left[\Big(\bar{U}_t^n-\bar{U}_s^n\Big)^2\right]  &\leq&  C (t-s)^{2-\frac{2}{\alpha}}.
\end{eqnarray}
If we use the notation
$$ \bar{h}_0(\xi_x,\xi_y):=\bar{h}(\xi_x,\xi_y)-\mathbb{E}\left[\bar{h}(\xi_x,\xi_y)\right] $$
then we have
\begin{eqnarray*}
    \mathbb{E}\left[\Big(\bar{U}_t^n-\bar{U}_s^n\Big)^2\right] 
  &=& n^{-4\delta}\mathbb{E}\Bigg[\Bigg(\sum_{x,y}N_{\lfloor nt \rfloor}(x)\Big(N_{\lfloor nt \rfloor}(y)-N_{\lfloor ns \rfloor}(y)\Big)\bar h_0(\xi_x,\xi_y) \\
  && \ \ \ \ \ \ + \sum_{x,y}\Big(N_{\lfloor nt \rfloor}(x)-N_{\lfloor ns \rfloor}(x)\Big)N_{\lfloor ns \rfloor}(y)\bar h_0(\xi_x,\xi_y) \Bigg)^2\Bigg]\\
 &\leq& 2 n^{-4\delta} \mathbb{E}\Bigg[\Bigg(\sum_{x,y}N_{\lfloor nt \rfloor}(x)\Big(N_{\lfloor nt \rfloor}(y)-N_{\lfloor ns \rfloor}(y)\Big)\bar h_0(\xi_x,\xi_y) \Bigg)^2\Bigg]\\
  && + \ 2n^{-4\delta}\mathbb{E}\Bigg[\Bigg(\sum_{x,y}\Big(N_{\lfloor nt \rfloor}(x)-N_{\lfloor ns \rfloor}(x)\Big)N_{\lfloor ns \rfloor}(y)\bar h_0(\xi_x,\xi_y) \Bigg)^2\Bigg].
\end{eqnarray*}
We continue the computation with the first of the two terms.  In the following we condition with respect to 
$ {\mathcal G}=\sigma(S_n;n\in\mathbb{N}) $. 
We make use of the assumption $ h(x,x)=0 $ and the fact that if 
$ x,y,u,v $ are all distinct then $ \bar{h}_0(\xi_x,\xi_y) $ and $ \bar{h}_0(\xi_u,\xi_v) $ are independent and centered and
we write
\begin{eqnarray*}
  \mathbb{E}\Bigg[\Bigg(\sum_{x,y}N_{\lfloor nt \rfloor}(x)\Big(N_{\lfloor nt \rfloor}(y)-N_{\lfloor ns \rfloor}(y)\Big)\bar{h}_0(\xi_x,\xi_y)\Bigg)^2\Bigg|{\mathcal G}\Bigg] &\leq& A+B+C +D
\end{eqnarray*}
with
\begin{eqnarray*}
  A&:=& \sum_{x,y}N_{\lfloor nt \rfloor}^2(x)\Big(N_{\lfloor nt \rfloor}(y)-N_{\lfloor ns \rfloor}(y)\Big)^2\mathbb{E}\left[\bar{h}^2_0(\xi_1,\xi_2)\big|{\mathcal G}\right], \\
 B &:=&  \sum_{x,y,z}N_{\lfloor nt \rfloor}(x)N_{\lfloor nt \rfloor}(z)\Big(N_{\lfloor nt \rfloor}(y)-N_{\lfloor ns \rfloor}(y)\Big)^2\mathbb{E}\left[|\bar{h}_0(\xi_1,\xi_2)\bar{h}_0(\xi_2,\xi_3)|\big|{\mathcal G}\right],
\end{eqnarray*}
\begin{eqnarray*}
  C :=\sum_{x,y,z}N^2_{\lfloor nt \rfloor}(x)\Big(N_{\lfloor nt \rfloor}(y)-N_{\lfloor ns \rfloor}(y)\Big)\Big(N_{\lfloor nt \rfloor}(z)-N_{\lfloor ns \rfloor}(z)\Big)\mathbb{E}\left[|\bar{h}_0(\xi_1,\xi_2)\bar{h}_0(\xi_2,\xi_3)|\big|{\mathcal G}\right].
\end{eqnarray*}
and 
\begin{eqnarray*}
  D :=2\sum_{x,x',y}N_{\lfloor nt \rfloor}(x')N_{\lfloor nt \rfloor}(x)\Big(N_{\lfloor nt \rfloor}(y)-N_{\lfloor ns \rfloor}(y)\Big)\Big(N_{\lfloor nt \rfloor}(x)-N_{\lfloor ns \rfloor}(x)\Big)\mathbb{E}\left[|\bar{h}_0(\xi_1,\xi_2)\bar{h}_0(\xi_2,\xi_3)|\big|{\mathcal G}\right].
\end{eqnarray*}

The Markov property together with Lemma \ref{QuadratAbschaetzungen} and Lemma \ref{KovarianzAbschaetzung} below imply
\begin{eqnarray*}
\mathbb{E}[B] &\leq& T^2n^2\mathbb{E}\left[\sum_x N^2_{\lfloor nt\rfloor -\lfloor ns\rfloor}(x)\right]\cov\left(\bar{h}(\xi_1,\xi_2),\bar{h}(\xi_2,\xi_3)\right) \\
    &\leq& C'n^2n^{2-\frac{1}{\alpha}}(t-s)^{2-\frac{1}{\alpha}}n^{-\frac{3}{\alpha}+\frac{4}{\alpha\beta}}  \\
    &=& (t-s)^{2-\frac{1}{\alpha}}O(n^{4\delta}).
\end{eqnarray*}
Again we see
\begin{eqnarray*}
  \mathbb{E}\left[C\right] &=& (\lfloor nt \rfloor-\lfloor ns\rfloor)^2
   \mathbb{E}\left[\sum_xN^2_{\lfloor nt \rfloor}(x)\right]\cov\left(\bar{h}(\xi_1,\xi_2),\bar{h}(\xi_2,\xi_3)\right) \\
 &\leq& n^2(t-s)^2n^{2-\frac{1}{\alpha}} T^{2-\frac{1}{\alpha}}n^{-\frac{3}{\alpha}+\frac{4}{\alpha\beta}}  \\
  &=& (t-s)^2O(n^{4\delta}).
\end{eqnarray*}
Further, we have by Cauchy-Schwarz that 
\begin{eqnarray*}  \mathbb{E}\left[\sum_xN_{\lfloor nt \rfloor}(x)\Big(N_{\lfloor nt \rfloor}(x)-N_{\lfloor ns \rfloor}(x)\Big)\right] &\leq& 
        \left(  \mathbb{E}\left[\sum_xN^2_{\lfloor nt \rfloor}(x)\right] \mathbb{E}\left[ \sum_x\Big(N_{\lfloor nt \rfloor}(x)-N_{\lfloor ns \rfloor}(x)\Big)^2 \right]\right)^{\frac{1}{2}} \\
   &\leq&   C'(nt)^{1-\frac{1}{2\alpha}}(n(t-s))^{1-\frac{1}{2\alpha}}.
\end{eqnarray*}
Now Lemma \ref{KovarianzAbschaetzung} implies
\begin{eqnarray*}
  \mathbb{E}\left[D\right] 
&\leq& [nt]([nt)-[ns]) \mathbb{E}\left[\sum_xN_{\lfloor nt \rfloor}(x)\Big(N_{\lfloor nt \rfloor}(x)-N_{\lfloor ns \rfloor}(x)\Big)\right]
                                \cov\left(\bar{h}(\xi_1,\xi_2),\bar{h}(\xi_2,\xi_3)\right)  \\
   &\leq& C''n(n(t-s))(nt)^{1-\frac{1}{2\alpha}}(n(t-s))^{1-\frac{1}{2\alpha}}n^{-\frac{3}{\alpha}+\frac{4}{\alpha\beta}}
\end{eqnarray*}

For $ t-s<\kappa<1 $ this is smaller than $   C''(t-s)^{2-\frac{2}{\alpha}}O(n^{4\delta}) $.
Finally for $ A $, due to Lemma \ref{KovarianzAbschaetzung1} below, we have
\begin{eqnarray*}
   \mathbb{E}\left[A\right] &\leq& \sqrt{\mathbb{E}\left[\left(\sum_xN^2_{nt}(x)\right)^2\right]
        \mathbb{E}\left[\left(\sum_yN_{n(t-s)}^2(y)\right)^2\right]}\var\left(\bar h(\xi_1,\xi_2)\right) \\
  &\leq& C''\sqrt{O((tn)^{4-\frac{2}{\alpha}})O(((t-s)n)^{4-\frac{2}{\alpha}})}\mathbb E[(\bar h(\xi_1,\xi_2))^2] \\
 &\leq & C'''(t-s)^{2-\frac{1}{\alpha}}n^{4-\frac{2}{\alpha}}n^{-\frac{2}{\alpha}+\frac{4}{\alpha\beta}} \\
 &\leq& C'''(t-s)^{2-\frac{1}{\alpha}}n^{4\delta}.
\end{eqnarray*}
All those inequalities together prove that there exists some constant $ K>0 $ such that for $ (t-s)<\kappa<1 $ one has
$$     \mathbb{E}\left[\Big(\bar{U}_t^n-\bar{U}_s^n\Big)^2\right] \leq K(t-s)^{2-\frac{2}{\alpha}}  .$$
This finishes the tightness proof.
\end{proof}

\begin{lem} \label{KovarianzAbschaetzung}
There is some constant $ C>0 $ such that
$$  \left|\cov\left(\bar{h}(\xi_1,\xi_2),\bar{h}(\xi_1,\xi_3)\right)\right| \leq C'n^{-\frac{3}{\alpha}+\frac{4}{\alpha\beta}}  .$$
\end{lem}

\begin{proof}
We first do the case $ \beta<\frac{4}{3} $. 
Note that by Assumption 1 part (iv) for some $ \gamma>\frac{3\beta}{4} $ ($\gamma\ne 1$), we have
\begin{eqnarray*}
   \mathbb{E}\left[|\bar{h}(\xi_1,\xi_2)\bar{h}(\xi_1,\xi_3)|\right] 
 &=& \int_0^\infty\int_0^\infty \mathbb{P}\left(|\bar{h}(\xi_1,\xi_2)|>s, |\bar{h}(\xi_1,\xi_3)|>t\right) dsdt\\
  &=& \int_0^{\rho n^{\frac{2}{\alpha\beta}}}\int_0^{\rho n^{\frac{2}{\alpha\beta}}}
                       \mathbb{P}\left(|h(\xi_1,\xi_2)|>s, |h(\xi_1,\xi_3)|>t\right) dsdt\\
 &\leq&  \int_0^{\rho n^{\frac{2}{\alpha\beta}}}\int_0^{\rho n^{\frac{2}{\alpha\beta}}}
       C_0(\max(1,s)\max(1,t))^{-\gamma} dsdt \\
 &=&C_0\left(1+\int_1^{\rho n^{\frac{2}{\alpha\beta}}}t^{-\gamma} dt\right)^2 \\
&\leq& C_0\left(1+\frac{1}{1-\gamma}\left(\left(\rho n^{\frac{2}{\alpha\beta}}\right)^{1-\gamma}-1\right)\right)^2 \\
 &\leq &  C_0\left(1-\frac{1}{1-\gamma}+\frac{\rho^{1-\gamma}}{1-\gamma} n^{-\frac{3}{2\alpha}+\frac{2}{\alpha\beta}}\right)^2
   =O(n^{-\frac{3}{\alpha}+\frac{4}{\alpha\beta}})
\end{eqnarray*}
Due to Lemma \ref{majoesp} this implies 
$$  \left|\cov\left(\bar{h}(\xi_1,\xi_2),\bar{h}(\xi_1,\xi_3)\right)\right| = O(n^{-\frac{3}{\alpha}+\frac{4}{\alpha\beta}}) .$$

Now assume $ \beta\geq\frac{4}{3} $. By \eqref{SYM} and Item (vi) of Assumption 1, we have for
 $ M_n:=\rho n^{\frac{2}{\alpha\beta}} $ that
\begin{eqnarray*}
\left|\cov\left(\bar{h}(\xi_1,\xi_2),\bar{h}(\xi_1,\xi_3)\right)\right| 
 &=& \left|\cov\left(\mathbf{h}_{M_n}(\xi_1,\xi_2),\mathbf{h}_{M_n}(\xi_1,\xi_3)\right)\right|\\
&\le& \left|\mathbb E\left[\mathbf{h}_{M_n}(\xi_1,\xi_2)\mathbf{h}_{M_n}(\xi_1,\xi_3)\right]\right|
      +|\mathbb E[\mathbf{h}_{M_n}(\xi_1,\xi_2)]|^2 \\
 &\leq& O\left(n^{-\frac{4\theta'}{\alpha\beta}}\right)+O\left(n^{-\frac{4}{\alpha\beta}(\beta-1)}\right)\\
 &\leq& O\left(n^{-\frac{4}{\alpha\beta}\left(\frac {3\beta}4-1\right)}\right)\\
  &=& O(n^{-\frac{3}{\alpha}+\frac{4}{\alpha\beta}} )
\end{eqnarray*}
since $ \theta'>\frac{3\beta}{4}-1 $.
\end{proof}

\begin{lem} \label{KovarianzAbschaetzung1}
We have
$$  \mathbb E\left[\bar{h}(\xi_1,\xi_2))^2\right] =O\left(n^{-\frac{2}{\alpha}+\frac{4}{\alpha\beta}}\right)  .$$
\end{lem}

\begin{proof}
We have
\begin{eqnarray*}
\mathbb E\left[\bar{h}(\xi_1,\xi_2))^2\right] &=&\int_0^{\rho n^{\frac 2{\alpha\beta}}}\mathbb P(|\bar h(\xi_1,\xi_2)|^2\ge s)\, ds
=\int_0^{\rho n^{\frac 2{\alpha\beta}}}\mathbb P(|\bar h(\xi_1,\xi_2)|\ge u) 2u\, du\\
&=&O(n^{\frac 2{\alpha\beta}(2-\beta)}),
\end{eqnarray*}
since $2u\mathbb P(|\bar h(\xi_1,\xi_2)|\ge u)\sim 2(c_0+c_1)u^{1-\beta}$ as $u$ goes to infinity.
\end{proof}

\begin{appendix}

\section{Stochastic integral with respect to the L\'evy sheet $Z$}\label{calculsto}

In this section, following \cite{KhoshnevisanNualart}, we give a simple construction of stochastic integral with respect to the $\beta$-stable L\'evy sheet $Z$. 
In \cite{KhoshnevisanNualart}, Khoshnevisan and Nualart considered  general
L\'evy sheet with symmetric distributions. Therefore
their results apply to the $\beta$-stable L\'evy sheet $Z$ only if $c_0=c_1$.
Nevertheless, we will see that their construction is expansible when $c_0\ne c_1$.

Let us recall that $Z$ satisfies the following properties:
\begin{itemize}
 \item $Z_{0,0}=0$;
 \item for any family $(A_k=[a_k,b_k]\times[a'_k,b'_k])_k$ of pairwise disjoint rectangles (with $a_k<b_k$ and $a'_k<b'_k$),
   the family of increments $(Z_{b_k,b'_k}+Z_{a_k,a'_k}-Z_{a_k,b'_k}-Z_{b_k,a'_k})_k$ is a family of independent
   random variables;
 \item for any rectangle $A=[a,b]\times[a',b']$ (with $a<b$ and $a'<b'$), the characteristic function of the increment
  $Z_{b,b'}+Z_{a,a'}-Z_{a,b'}-Z_{b,a'}$ is $\Phi_{(c_0+c_1)\lambda(A),(c_0-c_1)\lambda(A),\beta}$, where $\lambda$ is
  the Lebesgue measure on $\mathbb R^2$ and where we used the notation introduced in (\ref{FonCar}).
\end{itemize}
For any rectangle $A=[a,b]\times[a',b']$ (with $a<b$ and $a'<b'$), 
we define the stochastic integral of $\mathbf 1_A$
with respect to the L\'evy process as the increment of $Z$ in this rectangle, i.e.
\begin{equation}\label{indicatrice}
\int_{\mathbb R^2}\mathbf 1_{A}\, dZ_{x,y} := Z_{b,b'}+Z_{a,a'}-Z_{a,b'}-Z_{b,a'}. 
\end{equation}
We extend this definition by linearity to any linear combination $H$ of such indicator functions.
Observe that, if $H=\sum_{j=1}^{\mu}h_j\mathbf{1}_{A_j}$ where $(A_j)_j$ is a family
of pairwise disjoint rectangles and where $h_j\in \mathbb R$, then the characteristic function of 
$\int_{\mathbb R^2}H(x,y)\, dZ_{x,y}$ is given by
\begin{eqnarray*}
\forall z\in\mathbb R,\ \ \ 
    {\mathbb E}\left[\exp\left(iz\int_{\mathbb R^2}H(x,y)\, dZ_{x,y}\right) \right]&=& 
    \prod_{j=1}^{\mu} {\mathbb E}\left[\exp\left(izh_j\int_{\mathbb R^2}\mathbf{1}_{A_j}(x,y)\, dZ_{x,y}\right)\right]\\
    &=&  \prod_{j=1}^{\mu} \Phi_{(c_0+c_1)\lambda(A_j),(c_0-c_1)\lambda(A_j),\beta}(zh_j)\\
    &=&\prod_{j=1}^{\mu} \Phi_{(c_0+c_1)|h_j|^\beta_+\lambda(A_j),(c_0-c_1) |h_j|^\beta_-\lambda(A_j),\beta}(z)\\
    &=&\Phi_{(c_0+c_1)\sum_{j=1}^\mu |h_j|^\beta_+\lambda(A_j),(c_0-c_1)\sum_{j=1}^\mu |h_j|^\beta_-\lambda(A_j),\beta}(z)
\end{eqnarray*}
and so by
\begin{equation}\label{FCsimple}
\forall z\in\mathbb R,\ \ \ 
    {\mathbb E}\left[\exp\left(iz\int_{\mathbb R^2}H(x,y)\, dZ_{x,y}\right) \right]= 
    \Phi_{(c_0+c_1)\int_{\mathbb R^2}|H(x,y)|^\beta_+\, dxdy,(c_0-c_1)\int_{\mathbb R^2}|H(x,y)|^\beta_-\, dxdy,\beta}(z)).
\end{equation}

\begin{prop}(see \cite{KhoshnevisanNualart}\label{LEMappend1})\label{intsto}
Let H be a continuous compactly supported function from $\mathbb R^2$ to $\mathbb R$.
Let $(H_n)_n$ be a sequence of linear combination of indicators over rectangles converging pointwise to
$H$. Assume moreover that $(H_n)_n$ is a family of uniformly bounded functions with support in a same compact.
Then the sequence $\left(\int_{\mathbb R^2}H_n(x,y)\, dZ(x,y)\right)_n$ converges in probability
to a random variable with characteristic function $\Phi_{(c_0+c_1)\int_{\mathbb R^2}|H(x,y)|^\beta_+\, dxdy,(c_0-c_1)\int_{\mathbb R^2}|H(x,y)|^\beta_-\, dxdy,\beta}$.
\end{prop}

For a continuous compactly supported $H:\mathbb R^2\rightarrow\mathbb R$, 
we define $\int_{\mathbb R^2}H(x,y)\, dZ(x,y)$ as the limit in probability
given by Proposition \ref{intsto}
(observe that the limit does not depend on the choice of $(H_n)_n$).
\begin{proof}[Proof of Proposition \ref{LEMappend1}]
To prove the convergence in probability, it is enough to prove that
\begin{equation}\label{Cauchy}
\forall z\in\mathbb R,\ \ \ 
  \lim_{n,m\rightarrow +\infty}\mathbb E\left[\exp\left(iz\int_{\mathbb R}(H_n(x,y)-H_m(x,y))\, dZ_{x,y}\right)\right]=1.
\end{equation}
Observe that, for every real number $z$, we have
\begin{eqnarray*}
  && \left|\mathbb E\left[\exp\left(iz\int_{{\mathbb R}^2}(H_n(x,y)-H_m(x,y))\, dZ_{x,y}\right)\right]-1\right| \\
 &=&\left|\Phi_{(c_0+c_1)\int_{{\mathbb R}^2} |H_n(x,y)-H_m(x,y)|_+^\beta\, dxdy, (c_0-c_1)\int_{{\mathbb R}^2} |H_n(x,y)-H_m(x,y)|_+^\beta\, dxdy,\beta}(z)-1\right|\\
 &\le& C\int_{{\mathbb R}^2}  |H_n(x,y)-H_m(x,y)|^\beta\, dxdy(|c_0+c_1|+|c_0-c_1|)|z|^\beta,
\end{eqnarray*}
using the fact that $|e^{-a+ib}-e^{-a'+ib'}|\le|a-a'|+|b-b'|$ for any
real numbers $a,b,a',b'$ such that $a>0$ and $a'>0$.
Since $(H_n)_n$ converges pointwise and is uniformly bounded, we obtain (\ref{Cauchy})
by the Lebesgue dominated convergence theorem
(recall that $(H_n)_n$ is a sequence of uniformly bounded functions supported in a same compact).
Now the characteristic function  of the limit in probability $\int_{\mathbb R^2}H(x,y)\, dZ(x,y)$
is given by
\begin{eqnarray*}
 \mathbb E\left[\exp\left(iz\int_{\mathbb R^2}H(x,y)\, dZ(x,y)\right)\right]&=&
 \lim_{n\rightarrow +\infty} \mathbb E\left[\exp\left(iz\int_{\mathbb R^2}H_n(x,y)\, dZ(x,y)\right)\right]\\
 &=&\lim_{n\rightarrow+\infty}
   \Phi_{(c_0+c_1)\int_{\mathbb R^2}|H_n(x,y)|^\beta_+\, dxdy,(c_0-c_1)\int_{\mathbb R^2}|H_n(x,y)|^\beta_-\, dxdy,\beta}(z))\\
   &=&\Phi_{(c_0+c_1)\int_{\mathbb R^2}|H(x,y)|^\beta_+\, dxdy,(c_0-c_1)\int_{\mathbb R^2}|H(x,y)|^\beta_-\, dxdy,\beta}(z)),
\end{eqnarray*}
for every real number $z$.
\end{proof}

\end{appendix}


\begin{thebibliography}{xxx}

\bibitem{BarbourEagleson}{\scshape A.D. Barbour, G.K. Eagleson}, Poisson convergence for dissociated statistics. 
{\slshape Journal of the Royal Statistical Society Ser. B} {\bfseries 46} (1984) 397-402.



\bibitem{Billingsley}{\scshape Billingsley}, {\slshape Convergence of probability measure}. Second edition. John Wiley \& Sons, Inc., New York, (1999). 

\bibitem{Bolthausen}{\scshape E. Bolthausen}, 
A central limit theorem for two-dimensional random walks in random sceneries. {\slshape Annals of Probability} {\bfseries 17} (1989) 108-115.

\bibitem{Borodin1}{\scshape A. N. Borodin}, A limit theorem for sums of independent random variables defined on a recurrent random walk. {\slshape Dokl. Akad. Nauk SSSR}  {\bfseries 246} (1979) 786--787 (in Russian).

\bibitem{Borodin2}{\scshape A. N. Borodin}, Limit theorems for sums of independent random variables defined on a transient random walk. 
 {\slshape  Zap. Nauchn. Sem. Leningrad. Otdel. Mat. Inst. Steklov. (LOMI)}  {\bfseries 85} (1979), 17-29, 237, 244 (in Russian).

\bibitem{cab}{\scshape P. Cabus, N. Guillotin-Plantard}, Functional limit theorems for U-statistics indexed by a random walk. {\slshape Stochastic Processes and their Applications} {\bfseries 101} (2002) 143-160.


\bibitem{cast}{\scshape F. Castell, N. Guillotin-Plantard, F. P\`ene}, Limit theorems for one and two-dimensional random walks in random scenery. {\slshape Ann. Inst. H. Poincar\'{e}}  {\bfseries 49} (2013) 506-528.

\bibitem{BFFN}{\scshape F. Castell, N. Guillotin-Plantard, F. P\`ene, Br. Schapira}, A local limit theorem for random walks in random scenery and on randomly oriented lattices. {\slshape Annals of  Probability}  {\bfseries 39} (2011) 2079-2118. 

\bibitem{Cerny}{\scshape J. Cern\'y}, 
Moments and distribution of the local time of a two-dimensional random walk. {\slshape Stochastic Processes and their Applications} {\bfseries 117 } (2007) 262-270.


\bibitem{DDMS}{\scshape A. Dabrowski, H. Dehling, T. Mikosch, O.Sh. Sharipov}, Poisson limits for U-statistics, {\slshape Stochastic Processes and their Applications} {\bfseries 99} (2002) 137-157.


\bibitem{DU}{\scshape G. Deligiannidis, S. Utev}, An asymptotic variance of the self-intersections of random walks, {\slshape Sib. Math. J.}  {\bfseries 52} (2011) 639--650.

\bibitem{Dudley}{\scshape R.M. Dudley}, Distances of probability measures and random variables, {\slshape Annals of Mathematical Statistics} {\bfseries 39} (1968) 1563-1572.

\bibitem{Feller}{\scshape W. Feller}, An introduction to probability theory and its applications. Vol. II. Second edition, John Wiley \& Sons, Inc., New-York, (1971).

\bibitem{BFMa}{\scshape B. Franke, F. P\`ene, M. Wendler}, Stable limit theorem for U-Statistic processes indexed by a random walk, arXiv:1212.2133.

\bibitem{guil}{\scshape N. Guillotin-Plantard, V. Ladret}, Limit theorems for U-statistics indexed by a one dimensional random walk, {\slshape ESAIM} {\bfseries 9} (2005) 95-115.



\bibitem{jainpruit83}{\scshape N. C. Jain, W. E. Pruitt}  Asymptotic behavior of the local time of a recurrent random walk. {\slshape Annals of Probability} {\bfseries 12} (1984) 64-85.

\bibitem{kest}{\scshape H. Kesten, F. Spitzer}, A limit theorem related to an new class of self similar processes, {\slshape Zeitschrift f\"ur Wahrscheinlichkeitstheorie und verwandte Gebiete} {\bfseries 50} (1979) 5-25.

\bibitem{KhoshnevisanNualart}{\scshape D. Khoshnevisan, E. Nualart}, Level sets of the stochastic wave equation driven by a symmetric L\'evy noise, {\slshape Bernoulli} {\bfseries 14} (2008) 899--925.


\bibitem{LeGallRosen}{\scshape J.F. Le Gall, J. Rosen}, The range of stable random walks. {\slshape Annals of Probability} {\bfseries 19} (1991) 650--705. 




\bibitem{Neveu1976}{\scshape J. Neveu}, Processus ponctuels, {\slshape Ecole d'\'et\'e de Probabilit\'e de Saint-Flour}, Lecture Notes in Mathematics,  {\bfseries 598}, Springer Verlag, Berlin, (1976).

\bibitem{Resnick87}{\scshape S. I. Resnick}, {\slshape  Extreme values, regular variation, and point processes}. Springer Verlag, New York (1987).

\bibitem{SmorodnitskyTaqqu}{\scshape  G. Samorodnitsky, M.S. Taqqu M.S.}, {\slshape Stable Non-Gaussian Random Processes}. Chapman \& Hall., 
 New York, (1994).

\bibitem{Spitzer}{\scshape F. Spitzer}, {\slshape Principles of Random Walks}. Van Nostrand, Princeton, NJ (1964). 




\end{thebibliography}
\end{document}